\documentclass[12pt]{article}
\usepackage{mathtools,amsmath,amssymb}
\usepackage[textwidth=6.3in,textheight=9in,top=2cm,bottom=2cm,left=2.5cm,right=2.5cm]{geometry} % page format 
\usepackage{booktabs,tabularx}
\usepackage[table]{xcolor}
\usepackage{color,graphicx}
\usepackage{subcaption}
\usepackage{float}
\usepackage{algorithm2e} 
\usepackage{hyperref}

\newcommand{\R}{\mathbb{R}}

\newcommand{\bu}{\boldsymbol u}
\newcommand{\bv}{\boldsymbol v}
\newcommand{\bz}{\boldsymbol z}

\newcommand{\be}{\boldsymbol e}

\newcommand{\bU}{\boldsymbol U}

\newcommand{\bP}{\boldsymbol P}
\newcommand{\br}{\boldsymbol r}
\newcommand{\bc}{\boldsymbol c}

\newcommand{\bvar}{\boldsymbol{\varphi}}
\newcommand{\bvarg}{\boldsymbol{\varphi}^g}
\newcommand{\tr}{\tilde{r}}

\newcommand{\tK}{\tilde{K}}
\newcommand{\di}{\mathrm d}
\newtheorem{theorem}{Theorem}

\newcounter{remark}
\def\theremark {\arabic{remark}}

\newtheorem{Proof}{Proof}

\newenvironment{proof}{\begin{Proof}\rm}{\hfill $\Box$ \end{Proof}}

\title{Computational study of Proper Orthogonal Decomposition methods for parametric approximations}
\author{Bosco
Garc\'{\i}a-Archilla\thanks{Departamento de Matem\'atica Aplicada
II, Universidad de Sevilla, Sevilla, Spain. Research is supported by
Spanish MCINYU under grants PID2021-123200NB-I00 and PID2022-136550NB-I00. (bosco@esi.us.es)}
\and Alicia Garc\'{\i}a-Mascaraque\thanks{Departamento de
Matem\'aticas, Universidad Aut\'onoma de Madrid, Spain. Research is supported
by predoctoral contract associated to PID2022-136550NB-I00 supported by MCIN/AEI/10.13039/501100011033 and FSE+ (alicia.garcia-mascaraque@uam.es)}
  \and Julia Novo\thanks{Departamento de
Matem\'aticas, Universidad Aut\'onoma de Madrid, Spain. Research is supported
by Spanish MINECO
under grant PID2022-136550NB-I00. (julia.novo@uam.es)}}

\begin{document}
\maketitle
\abstract{This paper studies the computational implementation of proper orthogonal decomposition reduced-order models (POD-ROMs) for evolutionary parametric time-dependent partial differential equations (PDEs). For a one-parameter model, many papers in the literature build the correlation matrix by projecting onto $L^2(\Omega)$ even though optimal pointwise error estimates are proved when projecting onto $H^1_0(\Omega)$. We compare both scenarios and observe the similar performance in practice. Additionally, to get the POD approximation it is necessary to compute the nonlinear term in the reduced equations. This requires a high computational cost that increases when the problem gets more complex. Numerical results in this paper show different approaches to address this issue. Discrete Empirical Interpolation Method (DEIM) is the most efficient approach. It reduces computational time by approximately half compared to computing the whole FEM formulation by a tensor construction while maintaining the same accuracy. For a multiparameter model, we analyze the new and standard method proposed in \cite{newmethod} using a two-dimensional Brusselator model to complement the results and support the error analysis with a more complex system.}

\bigskip

\section{Introduction}
\label{sec:1}
It is known that getting precise approximations for time dependent parametric partial differential equations may result in a remarkably high computational cost. Reduced-Order Methods (ROM) drastically cut down the time to execute the numerical simulations while preserving enough accuracy in the desired approximations. 

The idea of Reduced-Order Methods of  Proper Orthogonal Decomposition (POD) type is to compact the information from a set of approximations, called \textit{snapshots}, computed with a standard method, see~\cite{ku-Vol}. We will refer to that method as the Full Order Method (FOM). In our context, we will consider the finite element method (FEM) as the FOM. In terms of accuracy, the Proper Orthogonal Decomposition (POD) method is sufficiently accurate if its error is comparable to the FEM error for unknown parameters and time instants.

Using the snapshots and Singular Value Decomposition (SVD), an orthonormal POD-ROM basis is constructed projecting onto $L^2(\Omega)$ or $H^1_0(\Omega)$. The ROM approximation is then computed by projecting the weak formulation of the PDE onto the reduced space generated by the first $r$ functions of the basis (for any integer $r\ge 1$).  

For most of the POD-ROM methods studied in the literature, the set of snapshots is based on full order approximations at different time instants and parameter values. However, with this strategy we cannot obtain pointwise optimal error bounds~\cite{dqgarnovo}. A major focus when analyzing the error on POD-ROM methods is to obtain a pointwise in time error estimate. For non parametric partial differential equations, some approaches use first-order divided differences instead of solely the set of snapshots (\cite{traianwang},\cite{rubino_etal}, \cite{ku-Vol}, \cite{singler}). 

Another key decision in constructing the POD-ROM is choosing the inner product space for the projection. While it is a common practice in the literature to project onto $L^2(\Omega)$ \cite{chapelle}, it has been proved that $H_0^1(\Omega)$ projections yield to optimal error bounds for the gradients that are directly bounded in terms of the tail of eigenvalues: $$\sum_{k>r} \lambda_k$$ where $\lambda_1 \geq \lambda_2 \geq \ldots \geq \lambda_{d_r} > 0$ are the nonzero eigenvalues of the dataset's correlation matrix, see \cite{rubino_etal}. In the present paper, we show numerical examples comparing both choices, projecting into $L^2(\Omega)$ or $H_0^1(\Omega)$, to understand deeply POD-ROM performance. 

We use the theoretical framework developed in~\cite{newmethod} and~\cite{temporal_nos} as the basis for our analysis. In~\cite{temporal_nos}, the continuous-in-time case is studied for both the FOM and POD-ROM methods. The data set from which the POD-ROM basis is constructed is examined for first order divided differences and continuous in time derivatives. We will just focus on the first order divided differences case. Optimal error bounds of the FOM and the POD-ROM solutions are proved for the $L^2(\Omega)$ norm for a semi-linear reaction-diffusion model problem. 

The ideas in~\cite{temporal_nos} are extended in~\cite{newmethod} for parametric dependent partial differential equations. In~\cite{newmethod}, a new method to get pointwise in time and in parameter error bounds is introduced. The error analysis of the standard method, based on several parameters, is also presented. The numerical experiments in~\cite{newmethod} are in 1d. In the present work, new numerical simulations will be carried out for one and several parameters for the two-dimensional Brusselator model \cite{dqgarnovo}, \cite{hairer}, \cite{hairerII}, \cite{homescu}, \cite{lefever}, \cite{roose} to show the performance of the approaches mentioned before. 

The most challenging part in incrementing the spatial dimension of the problem is dealing with the nonlinear term in the POD-ROM equations. When projecting a nonlinear term onto a low dimension POD-ROM subspace, the computational cost still scales with the size of the original full-order spatial mesh. If the main goal of using POD-ROM is to drastically reduce the computational cost by solving a reduced system, this dependence on the dimension of the FOM creates an inefficient online simulation of the POD-ROM approximations. To address this situation, we present two ways of solving the computational problem:
\begin{enumerate}
	\item Creating a tensor structure for the nonlinear term. We precompute a tensor by projecting the nonlinear term in the weak formulation into the POD basis. Because of that, the online part to get the POD approximation consists on a tensor-vector multiplication involving only the reduced coefficients. As it can be observed, there is no more dependence on the FOM dimension and using this method will lead to a faster online computation. The main disadvantage of this method is that it is limited to polynomial nonlinear formulations like the one in Brusselator equations~\cite{temporal_nos}.
	\item Using Discrete Empirical Interpolation Method (DEIM). DEIM is generally used in the literature when dealing with nonlinear terms in POD approximations (e.g., \cite{deim}, \cite{barrault}). It requires building a POD basis from snapshots of only the nonlinear term and then, defining a subspace of spatial indices to generate the nonlinear function. Because DEIM interpolates the nonlinear term using only a few important nodes, it significantly lowers the online computational cost. Using DEIM approximation to the nonlinear term in the POD formulation adds an extra term in the error analysis.
\end{enumerate}

The outline of this paper is as follows. In Sect.~\ref{sec:pre}, we state some preliminaries and notation to define the context. Sect.~\ref{sec:pod} is dedicated to presenting the POD method for one-parameter, using first-order difference quotients to generate the POD basis. Projecting onto $L^2(\Omega)$ or $H_0^1(\Omega)$ is also considered. Additionally, it is shown the general approach to proceed for several parameters including the intuitive deduction of the method in~\cite{newmethod}. In Sect.~\ref{sec:nonlin}, we explain in detail how to handle the nonlinear term in the reduced equations. Finally, in Sect.~\ref{sec:numexp}, some numerical experiments show the computational complexity of the problem in 2d, contrasting with the simpler 1d case studied in~\cite{newmethod}.

\section{Preliminaries and notation}
\label{sec:pre}
As a model problem we consider the following reaction-diffusion equation 
\begin{equation}
	\label{eq:model}
	\begin{array}{rcll}
		u_t^\alpha(t,x)-\nu_\alpha\Delta u^\alpha(t,x)+g_\alpha(u^\alpha(t,x))& = &f_\alpha(t,x) &\quad (t,x)\in (0,T^\alpha]\times \Omega,\\
		u^\alpha(t,x)&=&0,& \quad  (t,x)\in (0,T^\alpha]\times \partial \Omega,\\
		u^\alpha(0,x)&=&u_0^\alpha(x),&\quad  x\in \Omega,
	\end{array}
\end{equation}
in a bounded domain $\Omega\subset \R^d$, $d\in \{1,2,3\}$, where $\alpha$ is a parameter, $\alpha\in [\alpha_0,\alpha_L]$, $\nu_\alpha>0$ is the diffusion parameter and $g_\alpha$ is a nonlinear smooth function. For simplicity, we assume $g_{\alpha}$ is Lipschitz continuous with Lipschitz constant $L>0$, i.e. assume that 
\begin{equation}
	\label{eq:L}
	|g_{\alpha}(s) - g_{\alpha}(t)|\leq L |s-t| \quad .
\end{equation} 
The diffusion parameter, $\nu_\alpha$, nonlinear
term, $g_\alpha$, forcing term, $f_\alpha$, and initial condition, $u_0^\alpha$, may or may not depend all of them on the parameter 
$\alpha$, but there is at least 
one of them that depends on $\alpha$. For simplicity here and in the sequel we have taken $t_0=0$ as initial time. 

Let $C_p$ be the constant in the Poincar\'e inequality
\begin{equation}
	\label{eq:poin}
	\|v\|_0\le C_p \|\nabla v\|_0,\quad v\in H_0^1(\Omega),
\end{equation}
where $\| \cdot \|_0$ denotes the standard $L^2(\Omega)$ norm. 

Let us denote by $X_h^k$ a finite element space based on continuous piecewise polynomials of degree $k$  defined
over a partition of $\Omega$ into simplices $K$ of diameter $h$ 
and by $V_h^k$ the  finite element space based on continuous piecewise  polynomials of degree $k$ that satisfies also the homogeneous Dirichlet boundary conditions of the problem.

The following inverse inequality holds for all $v_h\in X_h^l$ , see~\cite[Theorem 3.2.6]{Cia78},
\begin{equation}
	\label{cota_inv}
	\| v_{h} \|_{W^{m,p}(K)} \leq c_{\mathrm{inv}}
	h^{n-m-d\left(\frac{1}{q}-\frac{1}{p}\right)}
	\|v_{h}\|_{W^{n,q}(K)},
\end{equation}
$0\le n\le m\le 1$, $1\le q\le p\le \infty$.

In the sequel,  $I_h u \in X_h^k$ will denote
the Lagrange interpolant of a continuous function $u$. The following bound can be found in~\cite[Theorem 4.4.4]{brenner-scot}
\begin{equation}
	\label{cota_inter}
	|u-I_h u|_{W^{m,p}(K)}\le c_\text{\rm int} h^{n-m}|u|_{W^{n,p}(K)},\quad 0\le m\le n\le k+1,
\end{equation}
where $n>d/p$ when $1< p\le \infty$ and $n\ge d$ when $p=1$.

We  consider a smooth curve {(i.e., with sufficient bounded derivatives)} $\alpha\in [\alpha_0,\alpha_L]\mapsto T^\alpha\in(0,T]$, for $\alpha_0, \alpha_L$ given parameters.
Our aim is to approximate the solutions of~\eqref{eq:model} on given time intervals $[0,T^\alpha]$ 
for parameters $\alpha$ in the given set $[\alpha_0,\alpha_L]$. 

Let us assume we want to approximate~\eqref{eq:model} for a parameter value $\alpha$ in a given time interval $[0,T^{\alpha}]$. We consider the semi-discrete finite element approximation: Find $u_h^\alpha\ :\ (0,T^{\alpha}]\rightarrow V_h^k$ such that
\begin{equation}
	\label{eq:gal_semi}
	(u_{h,t}^\alpha,v_h)+\nu_\alpha(\nabla u_h^\alpha,\nabla v_h)+(g_\alpha(u_h^\alpha),v_h)=(f_\alpha,v_h),\quad \forall\ v_h\in V_h^k,
\end{equation}
with $u_h^{\alpha}(0)=I_h u_0^\alpha\in V_h^k.$ In~\eqref{eq:gal_semi} $u_{h,t}$ denotes the derivative of $u_h$ with respect to time. If the weak solution of~\eqref{eq:model} is sufficiently smooth, then the following error estimate is well known,
see for example \cite[Proof of Theorem 1]{bosco-titi-fem}, \cite[Theorem 14.1]{Thomee}, \cite[Lemma 4.2]{Wang}, 
\begin{equation}\label{cota_gal}
	\max_{0\le t\le T}\left(\|(u^\alpha-u_h^{\alpha})(t)\|_0+h\|(u^\alpha-u_h^{\alpha})(t)\|_1\right)\le C(u)h^{k+1},\quad \alpha\in[\alpha_0,\alpha_L],\ 0\le t\le T^\alpha,
\end{equation}
where $\|\cdot\|_1$ denotes the standard norm in $H^1(\Omega)$.

\section{Proper orthogonal decomposition}
\label{sec:pod}
In this section, we outline the standard approach for obtaining a POD approximation for a reaction-diffusion equation defined in~\eqref{eq:model} that depends on only one-parameter $\alpha \in \R$, as described in \cite{ku-Vol}. Its detailed error analysis can be found in~\cite{newmethod}. After the presentation of the initial framework, the idea to obtain pointwise-in-time approximations by using first-order divided differences will be showed (\cite{temporal_nos}, \cite{locke}). This is the most established method in the literature since it yields uniform errors between the POD-ROM approximation and the FEM approximation to~\eqref{eq:model}. We will also refer to the first-order divided differences as difference quotients (DQs) in the remainder of the work. Later on this section, we will present the case when our aim is not only to obtain an approximation for the solution to~\eqref{eq:model} for one-parameter $\alpha$, but also for any parameter in a given interval. We develop the idea of the standard method in this context and introduce the new method proposed in~\cite{newmethod}. Using this method allows us to get continuous error bounds in time and parameter. This means that the bounds are not only valid for parameters and time instants used for generating the POD basis of the method, but also for time and parameter values in between. This final \textit{a priori} error estimate will be included in this section to better understand the numerical results in Sect.~\ref{sec:numexp}.

We start by defining the general way to approximate \eqref{eq:model} by the POD standard method \cite{ku-Vol}. Fix $\alpha \in \R$ and let $u^{\alpha}$ be the solution of the model problem in \eqref{eq:model} for that specific parameter. Taking $T=T^{\alpha}>0$ as the final time, we define the time step $\Delta t = T/M$, for $M>0$ fixed. Let $t_j =j\Delta t$, $j=0, \ldots, M$ be the uniformly distributed instant times. We define $\bU = {\rm span}\{\mathcal{U}\}$ such that		
\begin{equation}\label{eq:bU_std}
	\mathcal{U} = \left\{ u_h^{\alpha}(t_0), u_h^{\alpha}(t_1), \ldots, u_h^{\alpha}(t_M) \right\}.
\end{equation}
We can observe that $\mathcal{U}$ has $N=(M+1)$ elements, and it corresponds to the snapshots at $M+1$ uniformly distributed times. We can rewrite $\mathcal{U}$ as
$$ \mathcal{U} = \{ y_h^1, \ldots, y_h^N\},$$

to simplify the correlation matrix expression as $K=((k_{i,j}))\in \R^{N\times N}$ with
$$ k_{i,j}=\frac{1}{N}(y_h^i,y_h^j)_X,\quad i,j=1,\ldots,N,
$$
where $(\cdot,\cdot)_X$ is the inner product in $X$ and $X$ can be $L^2(\Omega)$ or $H_0^1(\Omega)$. If we denote by $\lambda_1\ge \lambda_2 \ge \ldots\ge \lambda_{d_r}>0$ the positive eigenvalues of $K$ and by $\bv_1,\ldots,\bv_{\boldsymbol{d_r}}\in \R^N$ the associated eigenvectors, the orthonormal POD basis functions of $\bU$ are
$$
\varphi_k=\frac{1}{\sqrt{N}}\frac{1}{\sqrt{\lambda_k}}\sum_{j=1}^Nv_k^j \, y_h^j,
$$
where $v_k^j$ is the $j$ component of the eigenvector $\bv_k$. 
For any $1\le r\le d_r$ let 
$$\bU^r={\rm span}\left\{\varphi_1,\varphi_2,\ldots,\varphi_r\right\},$$
be the reduced POD space and $P^r_X:X(\Omega)\rightarrow \bU^r$ the $X$-orthogonal projection onto $\bU^r$. The error in the projection in average behaves as the tail of the eigenvalues of the correlation matrix for a concrete $r$,
\begin{equation}\label{eq:eigenval_tail}
	\frac{1}{N}\sum_{j=1}^N\|(I-P^r_X)y_h^j\|_X^2=\sum_{k={r+1}}^{d_r}\lambda_k \,.
\end{equation}
We refer to $r$ as the number of POD modes.
Instead of having this average in time error bound, we want to have a pointwise-in-time error bound at the selected times. Then, the bound can also be generalized for continuous-time using interpolation techniques as in \cite{temporal_nos}. The idea was introduced in \cite{locke} and \cite{temporal_nos}. We will present the general scheme by denoting  $u_j^\alpha \colon= u_h^\alpha(t_j)$, for $0\le j\le M$, and maintaining the uniform time step $\Delta t$ as in the standard method above. Then, we can observe that for each time in the mesh, $t_j$, the FEM approximation stated in \eqref{eq:gal_semi} can be written as a combination of the solutions to the Galerkin formulation in the previous time instants,  $t_0, t_1, \ldots , t_{j-1}$:
\begin{equation}\label{eq:qdexp}
	u_j^\alpha = u_0^\alpha + \Delta t\,\frac{u_1^\alpha-u_0^\alpha}{\Delta t} + \cdots + \Delta t\,\frac{u_j^\alpha - u_{j-1}^\alpha}{\Delta t}.
\end{equation}
The difference quotients are defined as:
$$	D^t u^\alpha_j=\frac{u^{\alpha}_j-u^{\alpha}_{j-1}}{\Delta t}.		$$
Then, if we define $\bU$ introducing DQs instead of pointwise values as in \eqref{eq:bU_std}, we have that 
\begin{equation}\label{eq:bU_qd}
	\bU = \operatorname{span}\Bigl\{\sqrt{M+1}\,u_0^\alpha, \, D^tu_j^\alpha, \, 1\leq j \leq M\Bigr\} .
\end{equation}

Notice that the number of elements is $N= M+1$, the same as for the standard method. The average error bound in \eqref{eq:eigenval_tail} for this scenario leads to the following result
\begin{equation}\label{eq:meanQD}
	\|(I-P^r_X) u^\alpha_0\|_X^2 +\frac{1}{M+1}\sum_{j=1}^{M+1}\|(I-P^r_X)D^t u_j^\alpha\|_X^2=\sum_{k={r+1}}^{d_r}\lambda_k ,
\end{equation}
where $\lambda_1 \ge \lambda_2 \ge \ldots \ge \lambda_{d_r} \ge 0 $ now denote the positive eigenvalues of the correlation matrix $K = ((k_ij)) \in \R^{N\times N}$ with
$$ k_{ij} = \frac{1}{N} (y_h^i, y_h^j)_X, \qquad i,j = 1, \ldots, N,$$ 
where 
$$ y_h^1 = \sqrt{M +1} u_h^\alpha(t_0), \qquad y_h^j = D^t u_h^\alpha(t_{j-1}) \quad j=2, \ldots, N.$$
In case $X=H_0^1(\Omega)$, we can write $u_h^\alpha(t_j)$ as a linear combination of the difference quotients and the initial condition to get the desired pointwise-in-time estimate in $L^2(\Omega)$, \cite[Lemma 2]{fd},
\begin{equation}\label{eq:pointL2}
	\max_{0\leq j \leq M}\ \left\|\left( I - P^r_{H_0^1}\right)u_h^\alpha(t_j)\right\|_0^2  \leq  \left(2 + 4T^2\right) C_p^2 \sum_{k={r+1}}^{d_r}\lambda_k,
\end{equation}
where $C_p$ is the constant in Poincar\'e inequality in \eqref{eq:poin}. Following the same argument used to prove \eqref{eq:pointL2}, we also have the following bound in $H_0^1(\Omega)$,
\begin{equation}\label{eq:ergradH1}
	\max_{0\leq j \leq M}\ \left\| \nabla \left(I - P^r_{H_0^1}\right) u_h^\alpha(t_j) \right\|_0^2  \leq  \left(2 + 4T^2\right) \sum_{k={r+1}}^{d_r}\lambda_k.
\end{equation}
Then, we have both error bounds for $L^2(\Omega)$ and $H_0^1(\Omega)$ in terms of the tail of the eigenvalues of the correlation matrix. 

Let us consider the case $X=L^2(\Omega)$. In this scenario, the bound \eqref{eq:pointL2} gives us
\begin{equation}\label{proyec_L2_er}
	\max_{0\leq j \leq M}\ \left\| \left(I - P^r_{L^2}\right) u_h^\alpha(t_j) \right\|_0^2  \leq  \left(2 + 4T^2\right) \sum_{k={r+1}}^{d_r}\hat{\lambda}_k,
\end{equation} 
where we have added a hat to the eigenvalues to distinguish them from the ones in $X=H_0^1(\Omega)$ case. For this choice of $X$, we apply the inverse inequality from \cite[Lemma 2]{ku-Vol} so that 
\begin{equation}\label{eq:ergradL2}
	\max_{0\leq j \leq M} \left\|\nabla \left(I -P^r_{L^2}\right) u_h^{\alpha}(t_j) \right\|_0^2 \leq \|S_r\|_2\left(2 + 4T^2\right) \sum\limits_{k=r+1}^{d_r} \hat{\lambda}_k,
\end{equation}
where $S_r$ is the stiffness matrix in the POD formulation. A direct estimation of $\|S_r\|_2$ taking into account the inverse inequality presented in \eqref{cota_inv} is $\|S_r\|\simeq O(h^{-2})$. However, the behavior of  $\|S_r\|$ is usually better in practice. Then, if $X=L^2(\Omega)$, the error bound of the gradients grows with the size of the stiffness matrix. However, eigenvalues are different for every $X$ choice. In practice, we have
\begin{equation}
	\sum\limits_{k=r+1}^{d_r} \hat{\lambda}_k \ll \sum\limits_{k=r+1}^{d_r} \lambda_k .
\end{equation}
So, using $X=L^2(\Omega)$ does not imply a worse performance on the POD approximation.

Once the choice on the projection space is done, we consider the following semi-discrete POD-ROM approximation to approach \eqref{eq:model}: Find $u_r^{\alpha}:(0,T^\alpha]\rightarrow \bU^r$
such that
\begin{equation}\label{eq:pod}
	(u_{r,t}^{\alpha},v_r)+\nu_\alpha(\nabla u_r^{\alpha},\nabla v_r)+(g_\alpha(u_r^{\alpha}),v_r)=(f_\alpha,v_r),\quad \forall\ v_r\in \bU^r,
\end{equation}
with $u_r^{\alpha}(0)=u_{r,0}\in \bU^r$ and $u_{r,0}\approx u_0^\alpha$, $\alpha\in \R$.

In Sect.~\ref{sec:numexp}, we will carry out numerical experiments on the choice of the projection for the QDs approach. Here, we present \cite[Theorem 3]{temporal_nos} in order to be able to analyze the errors.
\begin{theorem}\label{theo:timepoint}	
	Let $u_h$ be the FEM approximation and $u_r$ the POD approximation to \eqref{eq:gal_semi}. Let $P^r u_h$ be the $H_0^1$-orthogonal projection onto $\bU^r$ and $L>0$ is the Lipschitz constant on the nonlinear term $g_\alpha$.
	Then, there exists a constant $C>0$ such that the following bound holds
	
	\begin{align*}
		\max _{0 \leq n \leq M}\left\|u_r\left(t_n\right)-u_h\left(t_n\right)\right\|_0^2 \leq & C \sum_{k=r+1}^{d_r} \lambda_k  \\
		& +C(\Delta t)^{2 q} \int_0^T\left(\left\|\frac{\partial^q \nabla u_h(t)}{\partial t^q}\right\|_0^2+\left\|\frac{\partial^{q+1} \nabla u_h(t)}{\partial t^{q+1}}\right\|_0^2\right) \di t
	\end{align*}
	
	Moreover, for all $t \in[0, T]$ assuming also $\frac{\partial^q u_h}{\partial t^q} \in L^{\infty}\left([0, T] ; H^1\right), q \geq 2$, it holds that
	
	\begin{align*}
		\left\|u_r(t)-u_h(t)\right\|_0^2 \leq & C\left(\sum_{k=r+1}^{d_r} \lambda_k+(\Delta t)^{2 q} \int_0^T\left\|\frac{\partial^{q+1} \nabla u_h(t)}{\partial t^{q+1}}\right\|_0^2 \di t\right)  \\
		& +C(\Delta t)^{2 q} \max _{0 \leq t \leq T}\left\|\frac{\partial^q \nabla u_h(t)}{\partial t^q}\right\|_0^2 .
	\end{align*}
\end{theorem}

\subsection{Proper orthogonal decomposition for parametric approximations}
\label{sec:podparam}
We have presented standard and difference quotients POD approaches for a fixed parameter $\alpha \in \R$ in \eqref{eq:model}. The motivation for presenting the DQs case was to get pointwise-in-time error bounds. The optimal scenario is to have a POD method that allows us to approximate \eqref{eq:model} for every parameter in a given interval $\alpha \in [\alpha_0, \alpha_L] \subset \R$, and at any moment in the associated time interval for that parameter, $t \in [0, T^\alpha]$. However, optimal pointwise error bounds cannot be proved in this case.

The natural procedure to deal with this situation is to extend the standard method based on generating a space of snapshots as defined in \eqref{eq:bU_std} for a selected number of parameters in the interval, $\{ \alpha_0, \alpha_1, \ldots, \alpha_L\} \subset [\alpha_0, \alpha_L]$. However, as in the case of only one parameter, for this method, optimal pointwise bounds cannot be proved.

To be able to prove pointwise bounds, in \cite{newmethod} a new POD method for continuous parameterized equations is proposed. In this section, we will briefly present the key ideas for the standard general method and the new method. We will also introduce the main result in \cite{newmethod} that describes the pointwise in time and parameter error bound dependence on the tail of eigenvalues, the time step size, and the parameter step size. This statement will be useful to understand the results obtained from the numerical experiments in Sect.~\ref{sec:numexp}

First, let us fix $L>0$, and $\alpha_0<\alpha_L$, so that the interval $[\alpha_0, \alpha_L]$ is consistent. Then, denote $\Delta \alpha = (\alpha_L - \alpha_0)/L$ and set $\alpha_l = \alpha_0 + l\Delta \alpha$, $l=0, \ldots, L$, $L$ uniformly distributed parameters in the interval. For each $l=0, \ldots, L$ and final time $T^l = T^{\alpha_l}>0$ let $\Delta t_l = T^l/M$, for $M>0$ fixed. Notice that $M$ is the same for all $\alpha_l$. Let $t_j^l =j\Delta t_l$, $j=0, \ldots, M$ be the instant times from which we compute the snapshots. We can define $\bU = {\rm span}\{\mathcal{U}\}$ such that		
$$
\mathcal{U} = \left\{ u_h^{\alpha_l}(t_j), 0\leq l \leq L, 0 \leq j \leq M \right\} \, .
$$
Observe $\mathcal{U}$ has $N=(L+1)(M+1)$ elements. For consistency with the POD method notation, we rewrite, $\mathcal{U}$ as
$$ \mathcal{U} = \{ y_h^1, \ldots, y_h^N\} \, .$$
The correlation matrix then is $K=((k_{i,j}))\in \R^{N\times N}$ with
$$ k_{i,j}=\frac{1}{N}(\nabla y_h^i,\nabla y_h^j),\quad i,j=1,\ldots,N \, ,
$$
and $(\cdot,\cdot)$ the inner product in $L^2(\Omega)$. We are assuming the projection to be in $H_0^1(\Omega)$ norm since we showed before that this is the optimal norm for getting pointwise in time estimations in $H_0^1(\Omega)$.

Following the definitions established for POD methods, we denote by $\lambda_1\ge \lambda_2\ldots\ge \lambda_{\boldsymbol{d_r}}>0$ the positive eigenvalues of $K$ and
by $\bv_1,\ldots,\bv_{\boldsymbol{d_r}}\in \R^N$ the associated eigenvectors. For any $1\le r\le d_r$ let 
$$\bU^r={\rm span}\left\{\varphi_1,\varphi_2,\ldots,\varphi_r\right\},$$
and $P^r:H_0^1(\Omega)\rightarrow \bU^r$ the $H_0^1$-orthogonal projection onto $\bU^r$. Then, it holds
$$
\frac{1}{N}\sum_{j=1}^N\|\nabla (I-P^r)y_h^j\|_0^2= \frac{1}{N}\sum_{l=0}^{L}\sum_{j=0}^{M+1}\|\nabla (I-P^r)u_h^{\alpha_l}(t_j)\|_0^2 = \sum_{k={r+1}}^{d_r}\lambda_k \, .
$$
In order to obtain
$$
\|\nabla (I-P^r)u_h^{\alpha}(t)\|_0^2 \sim \sum_{k=r+1}^{d_r}\lambda_k \, ,
$$
for any parameter $\alpha \in [\alpha_0,\alpha_L]$ and any time $t \in [0,T^\alpha]$, the idea, proposed in \cite{newmethod}, is to define the FEM approximation at a given time instant and for a fixed parameter as a combination of first-order finite differences (or difference quotients) computed from previous time instants. More precisely, let
$$
u_j^{\alpha_l} \colon= u_h(t_j^l,\alpha_l), \qquad 0 \le j \le M \, ,
$$
for $l \in \{0,\ldots,L\}$. Then,
$$
u_j^{\alpha_l} = u_0^{\alpha_l} + \Delta t\,D^t u_1^{\alpha_l}+ \cdots + \Delta t\,D^t u_j^{\alpha_l} \, .
$$
Adding the difference quotients in the POD initial set will lead to pointwise-in-time error bounds as in \eqref{eq:pointL2}. To conclude with pointwise-in-parameter estimates, the idea in \cite{newmethod} is to rewrite the finite difference for $\alpha_l$ in time $j$ as a linear combination of the finite differences for $\alpha_i$, $0\leq i \leq l-1$, and time $0\leq s \leq j-1$. As initial example, we set $l=1$ and $j=1$ so that
\begin{align*}
	D^t u_1^{\alpha_1} & =  D^t u_1^{\alpha_1} \pm D^t u_1^{\alpha_0}\\
	& = D^t u_1^{\alpha_1} + \Delta \alpha \, \frac{D^t u_1^{\alpha_1} - D^t u_1^{\alpha_1}}{\Delta \alpha_0} \\
	& = D^t u_1^{\alpha_1} + \Delta \alpha \, 	D^\alpha D^t u^{\alpha_1}_1 \, .
\end{align*}
where we denote the finite differences in time and parameter as:
$$
D^\alpha D^t u^{\alpha_l}(t_j^l)= \frac{D^t u^{\alpha_l}(t_j^l)-D^t u^{\alpha_{l-1}}(t_{j}^{l-1})}{\Delta \alpha} \quad 0\leq l \leq L, \, 1 \leq j \leq M \, .$$
Then,
$$
u_j^{\alpha_1} = u_0^{\alpha_1} + \Delta t_1 \, \sum\limits_{s=1}^j D^t u_j^{\alpha_0} + \Delta t_1 \, \sum\limits_{s=1}^j \Delta \alpha \, D^\alpha D^t u_j^{\alpha_1} \, .
$$

In general, for any $l \in {0, \ldots, L}$ and any $t_j^l$,
$$
u_j^{\alpha_l} = u_0^{\alpha_l} + \Delta t_l \, \sum\limits_{s=1}^j D^t u_j^{\alpha_0} + \Delta t_l \, \sum\limits_{i=1}^{l} \sum\limits_{s=1}^j \Delta \alpha \, D^\alpha D^t u_j^{\alpha_i}\, .
$$
For that reason, we define the space
$
\bU= {\rm span}\left( {\cal U}\right)\, ,
$
where
\begin{eqnarray*}
	{\cal U} &=&\left\{\sqrt{N}u_h^{\alpha_l}(t_0), \ 0\le l\le L,\right.\\
	&&\sqrt{(L+1)}D^t u_h^{\alpha_0}(t_j^0),\ 1\le j\le M,\\
	&&\left. D^t D^\alpha u_h^{\alpha_l}(t_j^l),\ 1\le j\le M,\ 1\le l\le L\right\} \, .
\end{eqnarray*}
Then, 
\begin{align*}
	&\sum_{l=0}^L\bigl\|\nabla(I-P^r)u_h^{\alpha_l}(t_0)\bigr\|_0^2 
	+\frac{1}{M+1}\sum_{j=1}^M\bigl\|\nabla(I-P^r)D^t u_h^{\alpha_0}(t_j^0)\bigr\|_0^2\\
	+\;&\frac{1}{N}\sum_{j=1}^M\sum_{l=1}^L
	\bigl\|\nabla(I-P^r)D^t D^\alpha u_h^{\alpha_l}(t_j^l)\bigr\|_0^2
	\;=\;\sum_{k=r+1}^{d_r}\lambda_k \, ,
\end{align*}
yields and we can get pointwise-in-time and in-parameter estimates. Note both approaches for continuous parameterized equations methods use $N=(L+1)(M+1)$ elements to generate the correlation matrix, but differ to get optimal pointwise estimates.

The complete error analysis for both methods is detailed in \cite{newmethod}. For the standard method, quasi-optimal bounds are proved using the ideas in \cite{dqgarnovo}. The following theorem can be found in \cite[Theorem 2]{newmethod}. This proof is an extension for parameterized equations of the continuous in time case presented in Theorem \ref{theo:timepoint}. 
\begin{theorem}\label{theo:parampoint}
	Assume $u_r^{\alpha}(t_0)=P^ru_h(0)$. Let $t\in[0,T^\alpha]$ and $\alpha\in[\alpha_0,\alpha_L]$. Then, there exists a constant $C$
	depending on $T$, $\alpha_L-\alpha_0$ and $C_p$  such that the following bound holds for $q\ge2$, $m\ge2,$ whenever the functions $u_h^{\alpha}$ are smooth enough so that all the terms are well defined 
	\begin{eqnarray}\label{max_puntual_todo}
		\|u_r^{\alpha}(t)-u_h^{\alpha}(t)\|_0^2\le C_{m,q}
		\sum_{k={r+1}}^{d_r}\lambda_k+C_{1,u_h}(\Delta t)^{2q}+C_{2,u_h}(\Delta \alpha)^{2m-1} \, ,
	\end{eqnarray}
	where
	\begin{eqnarray*}
		C_{m,q}
		&=& C\bigl((m+1)+qm\bigr), \nonumber\\
		C_{1,u_h}
		&=& C\sum_{j=l}^{l+m-1}\max_{0\le t\le T^j}
		\left\|
		\frac{\partial^q\nabla u_h^{\alpha_j}(t)}
		{\partial t^q}
		\right\|_0^2
		\nonumber\\
		&&{}+C\sum_{j=l}^{l+m-1}
		\int_0^{T^j}
		\left(
		\left\|
		\frac{\partial^q\nabla u_h^{\alpha_j}(t)}
		{\partial t^q}
		\right\|_0^2
		+
		\left\|
		\frac{\partial^{q+1}\nabla u_h^{\alpha_j}(t)}
		{\partial t^{q+1}}
		\right\|_0^2
		\right)\,\di t,
		\nonumber\\
		C_{2,u_h}
		&=& C\int_{\alpha_l}^{\alpha_{l+m-1}}
		\int_0^{T_\mu}
		\sum_{i+j\le m}
		\left(
		\left\|
		\frac{\partial^{i+j+1}\nabla u_h^\mu(t)}
		{\partial t^{i+1}\partial\mu^j}
		\right\|_0^2
		+\left\|
		\frac{\partial^{i+j}\nabla u_h^\mu(t)}
		{\partial t^i\partial\mu^j}
		\right\|_0^2
		\right)\,\di t\,\di\mu
		\nonumber\\
		&&{}+C\int_{\alpha_l}^{\alpha_{l+m-1}}
		\sum_{i+j\le m}
		\left\|
		\frac{\partial^{i+j}\nabla u_h^\mu(T^\mu t/T^\alpha)}
		{\partial t^i\partial\mu^j}
		\right\|_0^2
		\,\di\mu.
	\end{eqnarray*}
	
\end{theorem}
With this last result, we have presented all the theoretical framework that we will use in Sect.~\ref{sec:numexp} for analyzing new numerical simulations. First, numerical experiments will be carried out for approximating \eqref{eq:model} for one parameter with the difference quotients POD approach and different projection choices. Then, the implementation will include several parameters in an interval using the standard and the new methods explained in this section.

\section{Nonlinear term in reduced equations}
\label{sec:nonlin}
For simplicity, we avoid the dependence on $\alpha$ in the equations. Let us remember the semi-discrete finite element approximation previously defined in \eqref{eq:gal_semi}: Find $u_h :\ (0,T^{\alpha}]\rightarrow V_h^k$ such that
\begin{equation}\label{eq:gal_semi_no-alpha}
	(u_{h,t},v_h)+\nu(\nabla u_h,\nabla v_h)+(g(u_h),v_h)=(f,v_h),\quad \forall\ v_h\in V_h^k \, ,
\end{equation}
with $u_h(0)=I_h u_0\in V_h^k.$ In order to understand the scale for the nonlinear term in the POD reduced equations, it will be convenient to formulate the problem in a matrix form. Taking $v_h$ as the finite element basis $\{ \phi_1, \phi_2, \ldots, \phi_{n_k}\}$, we obtain a system of ordinary differential equations (ODEs). The formulation of \eqref{eq:gal_semi_no-alpha} is equivalent to solve the following ODE:
\begin{equation}
	M_h \frac{\di \bu_h(t)}{\di t} + \nu S_h \bu_h(t) + G(u_h) = f_h \, ,
\end{equation}
where $\bu_h \in \R^{n_k}$ are the FEM coefficients of $u_h$, $M_h \in \R^{n_k \times n_k}$ is the mass matrix, $S_h \in \R^{n_k \times n_k} $ is the stiffness matrix and $f_h \in \R^{n_k}$ is the constant force vector with $f_h^i = \left(f, \phi_i\right), \, i=1, \ldots, n_k$. $G(u_h)$ is the nonlinear term such that
\begin{equation}\label{eq:nonlin}
	G(u_h) = \begin{bmatrix}
		(g(u_h), \phi_1) \\
		\vdots \\
		(g(u_h), \phi_{n_k}) \\
	\end{bmatrix}\, .
\end{equation}
In the POD basis approach, the reduced space $ \bU^r = {\rm span}\{\varphi_1, \varphi_2, \ldots, \varphi_r\}\subset V_h^k$ is determined by the fix number of modes $r$. Remember, as stated before in \eqref{eq:pod}, that the semi-discrete POD-ROM approximation consists on finding $u_r : [0,T] \mapsto \bU^r$ such that 
\begin{equation}\label{eq:pod_no-alpha}
	(u_{r,t}, v_r) + \nu (\nabla u_r, \nabla v_r) + (g(u_r),v_r) =(f, v_r),\quad \forall\ v_r \in \bU^r \, .
\end{equation}
As we want to study the implementation of the reduced nonlinear term, we denote  
\begin{equation}\label{eq:ur-coef}
	u_r(t,x) = \sum_{j=1}^r c_j(t)\varphi_j(x) \, ,
\end{equation} 
where we define $\bc(t)  = [c_1(t), c_2(t), \ldots, c_{r}(t)]^T$ as the POD coefficients. Observe that $M_r \in \mathbb{R}^{r\times r}$ is the reduced mass matrix, $S_r \in \R^{r \times r}$ the reduced stiffness matrix, and $f_r \in \R^{r}$ such that $f_r^j = (f_{\alpha}, \varphi_j), \, j=1, \ldots, r$, is the reduced constant force term. The POD formulation above is equivalent to solve the following ODE:
\begin{equation}\label{eq:pod-matrix}
	M_r \frac{\di \bc(t)}{\di t} + \nu S_r \bc(t) + G_r(u_r) = f_r \, ,
\end{equation}
where
\begin{equation}\label{eq:nonlin-pod}
	G_r(c) = \begin{bmatrix}
		(g(u_r), \varphi_1) \\
		\vdots \\
		(g(u_r), \varphi_{r}) \\
	\end{bmatrix}\, .
\end{equation}
$$ 	$$
Observe that for a fixed $j=1, \ldots, r$, we have 
$$\left( g(u_r), \varphi_j \right)= \left( g(u_r) , \sum\limits_{i=1}^{n_k} \varphi_j(x_i) \phi_i \right)= \sum\limits_{i=1}^{n_k} \varphi_j(x_i) (g(u_r), \phi_i) \, ,$$
where $x_1, \ldots, x_{n_k}$ denote the nodes of the triangulation in the FEM space. Then, the nonlinear term in the reduced equations can be computed using the nonlinear term in the FEM formulation in \eqref{eq:nonlin} by
$$ G_r^j(u_r) = [\varphi_j(x_1),\ldots, \varphi_j(x_{n_k})]^T G(u_r) \, .$$
Note that $ G(u_r) \in \R^{n_k \times 1}$. It is obvious that using this procedure to compute the nonlinear term, the online phase still relies on $n_k$, the number of nodes in the FEM partition. This value is really large in order to provide accurate approximations in problems in high dimensions. Therefore, the POD approach has a high computational cost loosing efficiency.

We suggest two strategies to address this issue. To maintain only calculations based on the reduced coefficients in the online part, the first approach is to compute the nonlinear term in the FEM formulation in the offline code structure with a tensor structure, see Sect.~\ref{sec:nonlin-tensor}. Notice that this method can only be implemented if the nonlinear term has a polynomial structure. The second one corresponds to a generalized method used in the literature: the DEIM approximation. In Sect.~\ref{sec:nonlin-eim} the implementation of the method is described in detail.

In Sect.~\ref{sec:numexp}, we will compare, for the Brusselator model in two dimensions, the efficiency of both methods. We will also check the good performance of both procedures compared to computing the nonlinear term using the FEM structure.

\subsection{Nonlinear term via tensor structure}
\label{sec:nonlin-tensor}
As presented before, in order to get POD approximations we want to work only with the reduced coefficients in the POD basis. That is, we want the reduced system described in \eqref{eq:pod-matrix} to depend only on $r$. In this section, we present a tensor structure for polynomial nonlinear terms to achieve this goal. In particular, we start by showing the implementation of this tensor when $g(u)= u^2$. Then, we present the computation of the tensor structure for the nonlinear term in Brusselator equations.

Take $g(u) = u^2$, from \eqref{eq:ur-coef} we define $u_r = \sum\limits_{i=1}^r c_i\varphi_i$. Then, $$g(u_r) =  \left( \sum\limits_{j=1}^r c_j\varphi_j \right) \left( \sum\limits_{k=1}^r c_k\varphi_k \right)$$ and from \eqref{eq:nonlin-pod} we get
\begin{equation*}
	G_r(u_r) = \sum_{j,k = 1}^r c_j c_k \begin{bmatrix}
		(\varphi_j \cdot \varphi_k, \varphi_1) \\
		\vdots \\
		(\varphi_j \cdot \varphi_k, \varphi_r)
	\end{bmatrix} \, .
\end{equation*}
Thus, one can compute the tensor
\begin{equation*}
	T = ((\varphi_j \cdot \varphi_k, \varphi_i))_{1\leq i,j,k \leq r} 
\end{equation*}
and evaluate the nonlinear term with the {\sc MATLAB}'s command \texttt{tensorprod} in the following way
\begin{verbatim}
	tensorprod(T, c, 3, 1) * c;
\end{verbatim}
in this context, the coefficients are summed over the third index, that is $\sum\limits_{k=1}^r c_k T_{ijk}$, and then multiplied by the coefficients again so that 
\begin{equation*}
	G(u_r) = \sum_{j,k = 1}^r c_j  c_k T_{ijk} \,  .
\end{equation*}
We proceed in a similar way for the nonlinear term of Brusselator equations. For this case, we have a coupled solution $\bu = [u, v]^T$ and $g(\bu) = [u^2 v, -u^2v]^T$. Then,
\begin{equation*}
	G(\bu_r) = \begin{bmatrix}
		\left( \begin{bmatrix}u_r^2 v_r \\ -u_r^2 v_r \end{bmatrix}, \varphi_1 \right) \\
		\vdots \\
		\left( \begin{bmatrix}u_r^2 v_r \\ -u_r^2 v_r \end{bmatrix}, \varphi_r \right)
	\end{bmatrix} \, ,
\end{equation*} 
and $\bu_r(t,x) = \sum\limits_{i=1}^r c_i(t)\varphi_i(x)$. So that we can represent the POD basis in terms of a coupled basis for each component, 
$\varphi_i = \begin{bmatrix} \varphi_i^u \\ \varphi_i^v \end{bmatrix}$. Note that $\varphi_i^u, \varphi_i^v \in V_h^k$. The nonlinear term above can be determined by a tensor structure
\begin{equation*}
	T = \left( \left( \begin{bmatrix}
		\varphi_j^u\cdot \varphi_k^u \cdot \varphi_l^v \\
		- \varphi_j^u\cdot \varphi_k^u \cdot \varphi_l^v
	\end{bmatrix}, \varphi_i \right) \right)_{1\leq i,j,k,l \leq r} \, ,
\end{equation*}
such that
\begin{equation*}
	G(\bu_r) = \sum_{j,k,l} c_j c_k c_l T_{ijkl} \, .
\end{equation*} 
Using {\sc MATLAB}'s command \texttt{tensorprod}, the final computation for the Brusselator in the online phase corresponds to 
\begin{verbatim}
	tensorprod(tensorprod(T, c, 3, 1), c, 3, 1) * c;
\end{verbatim}
Using this method has two significant advantages. Firstly, if there is a dependence on the nonlinear term on a parameter $\alpha$, a general tensor with no dependence in the parameter can be implemented easily. This implementation is used in the case we want to get approximations for several parameters using a unique tensor structure computed only once in the offline phase. The second advantage is that the tensor structure only needs to be computed once for a sufficiently large number of modes $R$. For POD computations that takes $r \leq R$, the same tensor structure can be used just selecting the corresponding modes for the tensor-vector multiplication. This is convenient if we want to check the accuracy of the POD method for different values of $r$.

\subsection{Nonlinear term via DEIM}\label{sec:nonlin-eim}
In the POD-DEIM approach, instead of one POD basis, two POD basis are derived. The first one consists on the POD approach defined in any of the cases in Sect.~\ref{sec:pod}, while the second consists of snapshots of the nonlinear term described by the Lipschitz function $g$ in \eqref{eq:L}. Taking $T>0$ as the final time in~\eqref{eq:model}, we define the time step $\Delta t = T/M$, for $M>0$ fixed. We determine $\tilde{\bU}={\rm span}\{\tilde{\mathcal{U}}\}$ such that
$$ \tilde{U} = \{ I_h g(u(t_0,x)), I_h g(u(t_1,x)), \ldots, I_h g(u(t_{M},x)\} \, .$$
Observe that we have $N=M+1$ snapshots of the Lagrange interpolant of the nonlinear function. The POD-DEIM reduced system is constructed as in Sect.~\ref{sec:pod} using the standard approach for these snapshots. If we rewrite $\tilde{\mathcal{U}}$ in the following way
$$ \tilde{\mathcal{U}} = \{ y_h^1, \ldots, y_h^{N}\} \, , $$ as we did before. The correlation matrix is $K=((k_{i,j}))\in \R^{N\times N}$ with
$$ k_{i,j}=\frac{1}{N}(\nabla y_h^i,\nabla y_h^j),\quad i,j=1,\ldots,N \, ,
$$
and $(\cdot,\cdot)$ the inner product in $L^2(\Omega)$.

Following the POD basis construction by the eigenvalues and the associated eigenvectors of $K$, we have that for any $1\le \tr \leq d_{\tr}$, the reduced space is
$$\tilde{\bU}^{\tr}={\rm span}\left\{\varphi_1^g,\varphi_2^g,\ldots,\varphi_{\tr}^g\right\},$$
and $P^{\tr}:H_0^1(\Omega)\rightarrow \tilde{\bU}^{\tr}$ the $H_0^1$-orthogonal projection onto $\tilde{\bU}^{\tr}$. As $\tilde{U}^r \subset V_h^k$ , DEIM-POD basis can also be expressed in terms of the FEM basis functions. To understand in a better way the algorithm, we define the change of basis matrix
\begin{equation}
	\label{eq:phig}
	\bvarg = \begin{bmatrix}
		\varphi_1^g(x_1) & \varphi_2^g(x_1) & \cdots & \varphi_{\tr}^g (x_1) \\
		\varphi_1^g(x_2) & \varphi_2^g(x_2) & \cdots & \varphi_{\tr}^g (x_2) \\
		\vdots 		  & \vdots & \ddots & \vdots \\
		\varphi_1^g(x_{n_k}) & \varphi_2^g(x_2) & \cdots & \varphi_{\tr}^g(x_{n_k})
	\end{bmatrix} \in \R^{ n_k \times \tr} \, ,
\end{equation}
where $x_1, \ldots, x_{n_k}$ are the FEM spatial nodes. Let us observe that the values at those nodes are the coordinates of any function in the standard FEM basis. For the general POD basis $\bU^r$, the change of basis matrix,  $\bvar \in \R^{n_k \times r}$, can be computed in the same way. Note that $\tr$ do not have to be the same as $r$. The POD-DEIM reduced system is constructed by projecting on the space spanned by the original POD basis $\bU^r$ and applying DEIM (Discrete Empirical Interpolation Method) approximation to the nonlinear function~\cite{deim}. Keeping the notation in matrix form in~\eqref{eq:pod-matrix}, the resulting reduced system is given  by
\begin{equation}\label{eq:pod-eim}
	M_r \frac{\di \bc(t)}{\di t}  + \nu S_r \bc + \hat{G}_r(u_r) = f_r \, ,
\end{equation}
where
\begin{equation}\label{eq:nonlin-pod-eim}
	\hat{G}_r(u_r) = \begin{bmatrix}
		(\hat{g}(u_r), \varphi_1) \\
		\vdots \\
		(\hat{g}(u_r), \varphi_{r}) \\
	\end{bmatrix}\,  .
\end{equation}
The discrete empirical interpolant of the nonlinear term $g$ is $\hat{g}$ where
\begin{equation}\label{hatg}
	\hat{g}(u_r(t,x))=\sum\limits_{k=1}^{\tr} \tilde{c}_k(t) \, \varphi_k^g(x) \,,
\end{equation}
and the DEIM coefficients are defined as
$$ \tilde{\bc} = (P^T \bvarg)^{-1}g(P^T \boldsymbol{\varphi} \, \bc) \, ,
$$
with $\bvarg \in \R^{n_k \times \tr}$ as stated in~\eqref{eq:phig}, $\bvar \in \R^{n_k \times r}$ the change of basis matrix of the original POD basis. $P \in \R^{n_k \times m}$ is a matrix whose $m$ columns come from the corresponding DEIM indices $\left\{\rho_1,\ldots,\rho_m\right\}$ chosen from the following algorithm (see~\cite[Algorithm 1]{deim}):

\begin{algorithm}
	\caption{DEIM for the reduced nonlinear term}
	\label{alg:deim}
	
	\KwIn{$\bvarg \in \R^{n_k \times \tr}$, $m \leq \tr$}
	\KwOut{$\left\{\rho_1,\rho_2,\ldots,\rho_{m}\right\}$.}
	
	$[|\rho|, \rho_1]=\max\left\{|\bvarg(:\, , 1)|\right\}$ \tcp*[r]{means $\rho_1$ is the index of the maximum component of $|\bvarg(:\, ,1)|$}
	$W=[\bvarg(:\,,1)]$, $\bP=[\be_{\rho_1}]$, $\overline \rho=[\rho_1]$\;
	
	\For{$l = 2$ \KwTo $\tr$}{
		Solve $(\bP W)\tilde{\bc}=\bP^{T}\bvarg(:\, ,l)$ for $\tilde{\bc}$\;
		Define $\br=\bvarg(:\, ,l)-W \tilde{\bc}$\;
		Compute $[|\rho|, \rho_l]=\max\left\{|\br|\right\}$\;
		$W=[W \ \bvarg(:\, ,l)]$, $\bP=[\bP \ \be_{\rho_l}]$, $\overline \rho=[\rho, \rho_l]$\;
	}
\end{algorithm}
Observe that we need to compute \eqref{eq:nonlin-pod-eim}. To this end, we may write
$$\left( \hat{g}(u_r), \varphi_j \right)= \left( \sum\limits_{k=1}^{\tr} \tilde{c}_k \, \varphi^g_k \, , \varphi_j  \right)= \sum\limits_{k=1}^{\tr} \tilde{c}_k \, \left(  \varphi^g_k \, , \varphi_j \right) \quad j =1, \ldots r \, .$$

In practice we compute a matrix $Z\in \R^{\tr \times r}$ with $ Z_{kj} = (\varphi_k^g \, , \varphi_j)$. The dependence on the number of FEM nodes in the nonlinear term is moved to the offline phase, such that the online phase consists on solving the ODE 
\begin{equation}\label{eq:pod-eim-c}
	M_r \frac{\di \bc(t)}{\di t}  + \nu S_r \bc + Z^T \, \tilde{\bc} \ = f_r \,  .
\end{equation}
Note that for computing $\tilde{c}$ we need to evaluate the nonlinear function only in the $m \leq \tr$ nodes determined by DEIM algorithm in the online phase, $P^T \bvar \, \bc$. In practice, we usually take $m = \tr$ and $\tr \leq r$ so that the computational cost of solving \eqref{eq:pod-eim-c} depends mainly on the original number of modes, $r$. In Sect.~\ref{sec:numexp}, numerical simulations implementing DEIM algorithm are carried out to check the well performance of the method.

Taking this into account, we can now affirm that the solution of the original system in \eqref{eq:model} is then approximated by the solution from POD-DEIM reduced system in \eqref{eq:pod-eim-c}. The accuracy of this approximation can be measured arguing as in \cite[Theorem 1]{temporal_nos}. For convenience to the reader, we include the proof of this theorem. After this, we will explain how to adapt the error analysis to the POD-DEIM method.

\begin{theorem}\label{th1} 
	Let $u_r$ be the POD-ROM approximation solving \eqref{eq:pod} and let $P^r u_h$ be the $H_0^1$-orthogonal projection  onto $\bU^r$ of the semi-discrete Galerkin approximation $u_h$ defined in \eqref{eq:gal_semi_no-alpha}.	Then, for the constant $K = \frac{2}{T} + 2L$, the following bound holds for all~$t\in [0,T]$
	\begin{eqnarray}\label{er_pro_ur}
		\lefteqn{\|u_r(t)-P^r u_h(t)\|_0^2+2\nu\int_0^t \|\nabla (u_r(s)-P^r u_h(s))\|_0^2 \ \di s}\nonumber\\
		&\le& e^{Kt}\|u_r(0)-P^r u_h(0)\|_0\\
		&&+ e^{KT}T\left(\int_0^t \|(I-P^r)u_{h,s}(s)\|_0^2\ \di s+L^2\int_0^t\|(I-P^r)u_h(s)\|_0^2 \ \di s\right).\nonumber
	\end{eqnarray}
\end{theorem}

\begin{proof}
	A straightforward calculation, using the projection property, shows 
	that the projection of $u_h$ satisfies for all $v_r\in \bU^r$
	\begin{eqnarray*}
		\lefteqn{(P^ru_{h,t},v_r)+\nu(\nabla P^ru_h,\nabla v_r)+(g(P^r u_h),v_r)}\\
		&=&(f,v_r)+\left((P^r-I) u_{h,t},v_r\right) +\left(g(P^r(u_h))-g(u_h),v_r\right).
	\end{eqnarray*}
	Denoting the error by
	$
	e_r=u_r-P^r u_h \in \bU^r
	$
	and subtracting this identity from \eqref{eq:pod} gives for all $v_r\in \bU^r$
	%	\begin{equation}\label{eq:laer}
		%		(e_{r,t},v_r)+\nu(\nabla e_r,\nabla v_r) =((I-P^r)u_{h,t},v_r)+(g(P ^r u_h)-g(u_r),v_r)
		%		+(g(u_h)-g(P^r u_h),v_r).
		%	\end{equation}
	\begin{eqnarray}\label{eq:laer}
		(e_{r,t},v_r)+\nu(\nabla e_r,\nabla v_r)
		&=&((I-P^r)u_{h,t},v_r)
		+(g(P^r u_h)-g(u_r),v_r)
		\nonumber\\
		&&{}+(g(u_h)-g(P^r u_h),v_r).
	\end{eqnarray}
	
	Taking $v_r=e_r$ in \eqref{eq:laer}, applying the Cauchy--Schwarz and Young inequalities together with
	the Lipschitz continuity of~$g$,
	we get
	\begin{eqnarray}\label{eq:th_er1}
		\lefteqn{ \frac{1}{2}\frac{d}{dt}\|e_r\|_0^2+\nu\|\nabla e_r\|_0^2 }\nonumber \\
		& \le & \frac{T}{2}\|(I-P^r)u_{h,t}\|_0^2
		+\frac{1}{2T}\|e_r\|_0^2+L\left\| e_r\right\|_0^2
		+L\left\|(I-P^r)u_h\right\|_0\left\|e_r\right\|_0\nonumber\\
		&\le&\frac{T}{2}\|(I-P^r)u_{h,t}\|_0^2
		+\frac{1}{2T}\|e_r\|_0^2+L\|e_r\|_0^2 
		+\frac{T}{2}L^2\|(I-P^r)u_h\|_0^2+\frac{1}{2T}\|e_r\|_0^2.
	\end{eqnarray}
	
	Multiplication by $2$ leads to 
	\[
	\frac{d}{dt}\|e_r\|_0^2+2\nu\|\nabla e_r\|_0^2 \le K\|e_r\|_0^2
	+T\|(I-P^r)u_{h,t}\|_0^2+TL^2\|(I-P^r)u_h\|_0^2 \, ,
	\]
	where
	\begin{equation}
		\label{constant_K}
		K=\frac{2}{T} + 2L \, .
	\end{equation}
	Integrating in time and applying Gronwall's lemma we reach \eqref{er_pro_ur}.
\end{proof}
We now explain how to bound the error in the POD-DEIM method. We first observe that POD-DEIM consists on finding $u_r : [0,T] \mapsto \bU^r$ such that 
\begin{equation*}
	(u_{r,t}, v_r) + \nu (\nabla u_r, \nabla v_r) + (\hat{g}(u_r),v_r) =(f, v_r),\quad \forall\ v_r \in \bU^r \, ,
\end{equation*}
where the nonlinear term $\hat g$ is defined in \eqref{hatg}.

Then, instead of the term $( g(P^r u_h) - g(u_r), e_r)$ in \eqref{eq:laer}, we have $$( g(P^r u_h) - \hat{g}(u_r), e_r).$$ To handle this term we add and subtract $(g(u_r), e_r)$ to get
\begin{equation*}
	(g(P^r u_h) - \hat{g}(u_r), e_r) =  (g(P^r u_h) - g(u_r), e_r) + (g(u_r) - \hat{g}(u_r), e_r) \, .
\end{equation*} 

Then, denoting the DEIM error as  
$$  e_r^{\text{DEIM}}(g) = g(u_r) - \hat{g}(u_r) \, ,$$
we obtain
\begin{equation*}
	(g(u_r) - \hat{g}(u_r), e_r) \leq \left\|e_r^{\text{DEIM}}(g) \right\|_0 \left\|e_r\right\|_0 \leq \frac{T}{2} \left\|e_r^{\text{DEIM}}(g) \right\|_0^2 + \frac{1}{2T}\left\| e_r \right\|_0^2 \, .
\end{equation*}
These two terms are added to equation \eqref{eq:th_er1}. Then we finally get for $\tK = \frac{3}{T} + 2L$ and for all $t \in [0,T]$,
\begin{eqnarray}\label{ladel_deim}
	\lefteqn{\|u_r(t)-P^r u_h(t)\|_0^2+2\nu\int_0^t \|\nabla (u_r(s)-P^r u_h(s))\|_0^2 \ \di s}\nonumber\\
	&\le& e^{\tK t}\|u_r(0)-P^r u_h(0)\|_0\nonumber\\
	&&+ e^{\tK T}T\left(\int_0^t \|(I-P^r)u_{h,s}(s)\|_0^2\ \di s+L^2\int_0^t\|(I-P^r)u_h(s)\|_0^2 \ \di s \right)\\
	&&+ e^{\tK T}T\int_0^t \| e_r^{\text{DEIM}}(g)\|_0^2\ \di s.\nonumber
\end{eqnarray}
We observe that the last term in \eqref{ladel_deim} is the extra error coming from using the DEIM algorithm in the computation of the nonlinear term (compared with \eqref{er_pro_ur}).

In \cite[Lemma 3.2]{deim} the DEIM error is bounded by  
$$  e_r^{\text{DEIM}}(g) \leq C \| (I -P^{\tr}) g \|_0^2 \, , $$
such that $P^{\tr}$ is the  $H_0^1$-orthogonal projection onto $\tilde{\bU}^{\tr}$, and
$$ C \leq \left( 1 + \sqrt{2 n_k}\right)^{m-1} \, ,$$
where $n_k$ is the number of nodes in the FEM and $m$ is the number of nodes used in the DEIM algorithm.
We observe that the upper bound for the constant $C$ obtained in \cite{deim} is not realistic since as $m\rightarrow n_k$ the empirical interpolant converges to the projection and, as a consequence, $C\rightarrow 1$.

\section{Numerical Experiments}\label{sec:numexp}
Let us consider the system of Brusselator with diffusion in two spatial dimensions:
\begin{equation}\label{bruss}
	\begin{array}{rclcl}
		u_t&=&\nu\Delta u + \alpha +u^2v  - (\beta + 1) u,&\qquad& (x,t)\in \Omega \times (0,T],\\
		v_t &=&\nu \Delta v+ \beta u -u^2v,& \qquad
		& (x,t)\in \Omega \times (0,T],\\
		%% && u(x,0)= u_0(x), \quad v(x,0)=v_0(x),&& x\in \Omega,\\
		&&u(x,t)=\alpha, \quad v(x,t) =\beta/\alpha,&& (x,t)\in \Gamma_1\times(0,T],\\
		&&\partial_n u(x,t)=\partial_n v(x,t) =0,&&  (x,t)\in \Gamma_2\times(0,T],\\
	\end{array}
\end{equation}
where $\nu$ is a positive parameter, $\Omega=[0,1]\times[0,1]$, $\Gamma_1\subset\partial\Omega$ is the union of sides $\{x=1\}\bigcup \{ y=1\}$, and~$\Gamma_2$ is the rest of~$\partial\Omega$. This system has an unstable equilibrium $u=\alpha$, $v=\beta/\alpha$, and, for $\nu$, sufficiently small, a stable limit cycle. If the boundary conditions are $\partial_n u(x,t)=\partial_n v(x,t) =0$ for $(x,t)\in(0,T]\times\partial\Omega$, then, the periodic orbit is flat in space, $\nabla u=\nabla v=0$, but with those in~\eqref{bruss}, it has more spatial complexity. 
In the sequel, we denote by~$\bu=(u,v)$ the two-component solution of~\eqref{bruss} in the variables~$u$ and~$v$.

We choose the value $\nu=0.002$, which offered an adequate balance between spatial complexity and computability in the two-dimensional scenario. In this case we consider a triangulation of~$\Omega$ based on a uniform mesh with 80 subdivision per side and with diagonals running southwest-northeast. We use a spatial discretization with quadratic finite elements, which resulted in a system with 51200 degrees of freedom. For this discretization, we compute for each choice of $(\alpha, \beta)$ in~\eqref{bruss} the corresponding limit cycle (see details in~\cite[Section~6]{temporal_nos}). 

In Sect.~\ref{sec:numone}, we introduce numerical simulations for one parameter in the model. In Sect.~\ref{sec:numparam}, we consider the case of several parameters.

\subsection{One-value parameter model}
\label{sec:numone}
Let us fix $\alpha=1$ and $\beta=2.75$ in \eqref{bruss}. The period for this values is $T=6.78570$s (up to five significant digits). In Fig.~\ref{fig:femsolution}, we present the FEM solution at different significant time instants to see the complexity of problem as time evolves. We consider $t=0$, $t\approx T/3$ and $t=T$. We also show the solution at $ t= 1.23918$ (up to five significant digits), where the maximum of $ \|\nabla u \|_0^2 + \|\nabla v\|_0^2 $  is attained, to see the large variations that occur over one period.

\begin{figure}[htbp]
	\centering
	\begin{minipage}{\textwidth}
		\centering
		\includegraphics[width=0.24\linewidth]{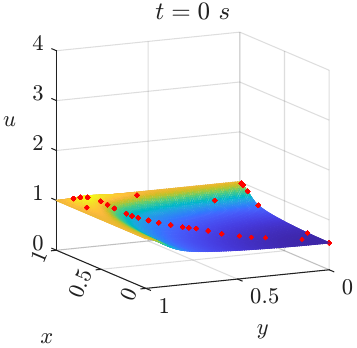}
		\includegraphics[width=0.24\linewidth]{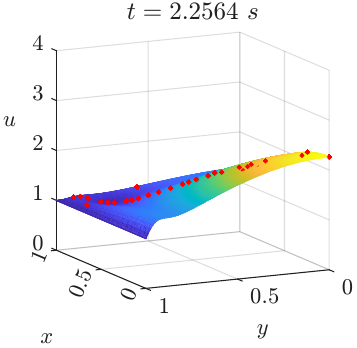} 
		\includegraphics[width=0.24\linewidth]{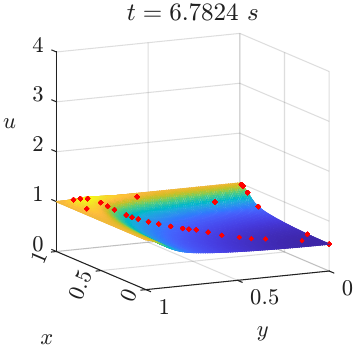}
		\includegraphics[width=0.24\linewidth]{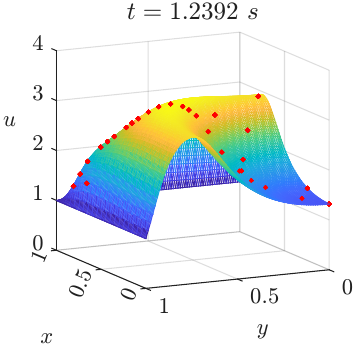}
	\end{minipage}
	
	\vspace{1cm} 
	
	\begin{minipage}{\textwidth}
		\centering
		\includegraphics[width=0.24\linewidth]{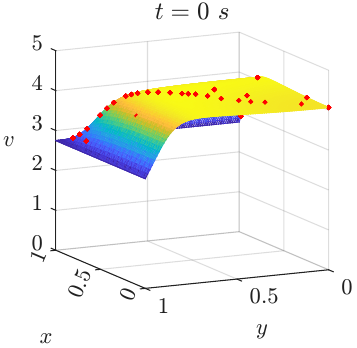}
		\includegraphics[width=0.24\linewidth]{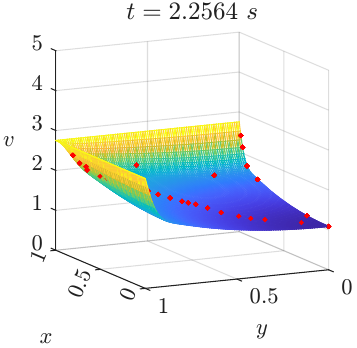} 
		\includegraphics[width=0.24\linewidth]{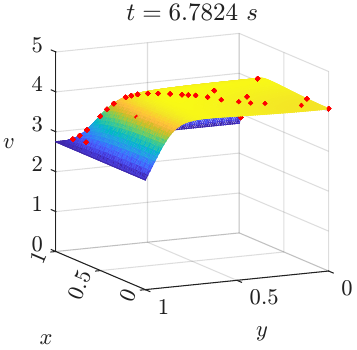}
		\includegraphics[width=0.24\linewidth]{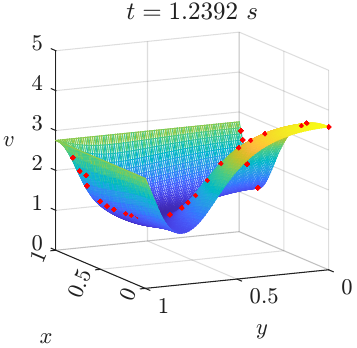}
		
	\end{minipage}
	\caption{Components $u$ (up) and $v$ (down) of the periodic solution at time $t=0, t=2.256378 \approx T/3, t=T$ and $t=1.23918$ (up to five significant digits), where the maximum of $ \|\nabla u \|_0^2 + \|\nabla v\|_0^2 $  is attained. Red points indicate the $m=30$ selected nodes by DEIM algorithm, see Algorithm \ref{alg:deim}.}
	\label{fig:femsolution}
\end{figure}

As the problem has non-homogeneous Dirichlet boundary conditions, we compute the POD basis $\bU^r$ with DQs plus the initial time snapshot subtracting the boundary-condition vector
\begin{equation}
	\label{eq:p}
	p = \begin{pmatrix}
		\alpha\,\mathbf{1}_{nn}\\[1mm]
		\beta/\alpha\,\mathbf{1}_{nn} 
	\end{pmatrix} \, ,
\end{equation} where $\mathbf{1}_{nn}\in\R^{nn}$ denotes the constant vector with $nn$, the number of nodes in the FEM mesh, ones. The DQs choice is done in order to have pointwise in time error estimates, see Sect.~\ref{sec:pod}. We consider both projections onto $X=L^2(\Omega)$ and $X=H_0^1(\Omega)$ to define the POD basis. We want to find the POD approximation $\bu_r = \bz_r + p$ such that $\bz_r:(0,T]\rightarrow \bU^r$ and it solves:
\begin{eqnarray}\label{eq:pod_ex}
	(\bz_{r,t},v_r)+\nu(\nabla \bz_r,\nabla v_r)+(g(\bz_r + p),v_r)=(f,v_r),\quad \forall\ v_r\in \bU^r,
\end{eqnarray}
with $\bz_r(0)=P^r_X (\bu^0 - p)$. We integrate in time \eqref{eq:pod_ex} with \texttt{ode15s} and a small enough tolerance so that the temporal errors are negligible, see \cite{NDF1}.

First, we want to understand the influence of the choice on the projection space $X$. In Fig.~\ref{fig:L2-H1}, we present singular values, $\sigma_k = \lambda_k^{1/2}$ for $X=H_0^1(\Omega)$ and $\hat{\sigma}_k = \hat{\lambda_k}^{1/2}$ for $X=L^2(\Omega)$ taking  $M=32\, , 64\, , 128\, , 256$ and $512$. Recall that bounds  \eqref{eq:ergradH1} and \eqref{eq:ergradL2} depends on the tail of the eigenvalues. It can be observed for $M=32$ and $M=64$ that the larger the number of snapshots, the slower the decay of the singular values. For $M \ge 128$ the results do not change. In this model it is sufficient to have $\Delta t = T/128$ in the POD method.

\begin{figure}[ht]
	\centering
	\includegraphics[width=0.48\textwidth]{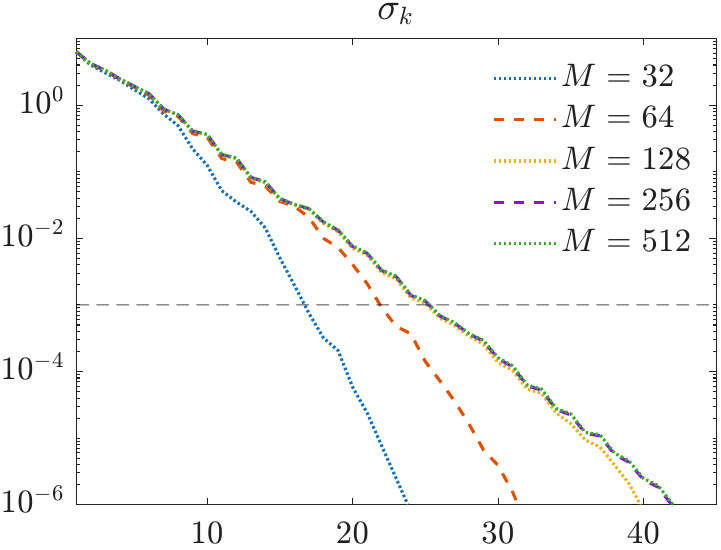}%
	\hfill
	\includegraphics[width=0.48\textwidth]{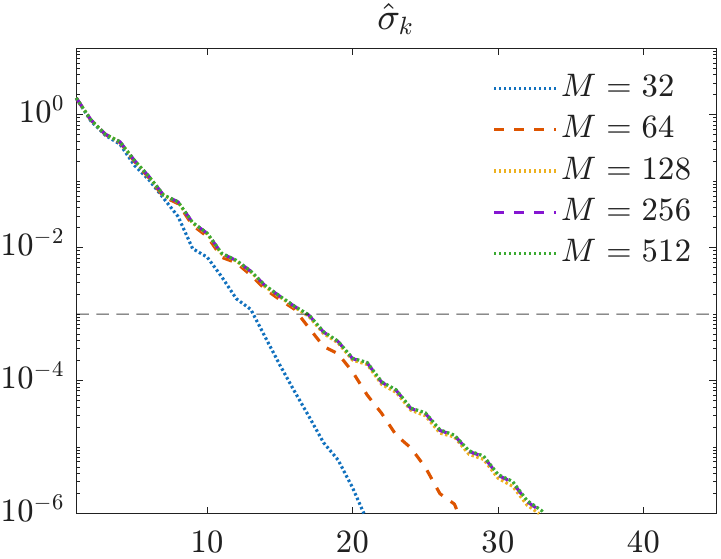}
	\caption{\label{fig:L2-H1} Singular values of the correlation matrix for $M=32, 64, 128, 256$ and $512$ projecting onto $X=H_0^1(\Omega)$ (on the left) and onto $X=L^2(\Omega)$ (on the right). The discontinuous black line represents the value $10^{-3}$.}
\end{figure}

As expected from the comments in Sect.~\ref{sec:pod}, if we choose $X=L^2(\Omega)$, we have smaller singular values than taking $X=H_0^1(\Omega)$. For $M= 128$, note that in order to have the tail of singular values below $10^{-3}$ (discontinuous line) with $X=L^2(\Omega)$ is enough to take $r=17$ while for $X=H_0^1(\Omega)$, we need to take $r=26$. 
However, observe that the value $r=17$ would give errors in $L^2$ of order $10^{-3}$ but larger errors in $H^1$. For this reason,
%Because of that, even if the optimal choice for pointwise in time error bounds is $X=H_0^1(\Omega)$, in practice $L^2(\Omega)$ performs also well. 
we take $r=26$ in both cases. We define
\begin{equation}
	\label{error_notation}
	{\boldsymbol \varepsilon}_r(t) = \left( \bu_h(t) - p \right) - P^r_X \left( \bu_h(t) - p \right),\qquad \be_r(t) = \bu_h(t)-\bu_r(t),
\end{equation}
the projection error in time and the error between the FEM and POD approximation, respectively. In Tab.~\ref{tab:ercont}, we present the relative errors of ${\boldsymbol \varepsilon}_r$ and $\be_r$ in $H^1$ and $L^2$ norm for $X=L^2(\Omega)$ and  $X=H_0^1(\Omega)$, measured on a very fine partition over a period (2049 equally-distributed points). For convenience, we omit the notation indicating relative errors, since here and in the sequel, all errors presented are relative.

\begin{table}[!t]
	\centering
	\begin{tabular}{|c|c|c|c|c|}
		\hline
		$X$ &
		$\displaystyle \max_{t}\|\nabla {\boldsymbol \varepsilon_r}(t)\|_0$ &
		$\displaystyle \max_{t}\|\nabla \be_r(t)\|_0$ &
		$\displaystyle \max_{t}\|{\boldsymbol \varepsilon_r}(t)\|_0$ &
		$\displaystyle \max_{t}\|\be_r(t)\|_0$ \\
		\hline
		$L^2(\Omega)$ &
		$1.1113 \times 10^{-4}$ &
		$9.8163 \times 10^{-5}$ &
		$2.5481 \times 10^{-6}$ &
		$4.1344 \times 10^{-6}$ \\
		\hline
		$H_0^1(\Omega)$ &
		$8.4756\times10^{-5}$ &
		$9.5666\times10^{-5}$ &
		$5.8134\times10^{-6}$ &
		$3.3514\times10^{-6}$ \\
		\hline
	\end{tabular}
	\caption{\label{tab:ercont} Maximum relative errors ${\boldsymbol\varepsilon}_r$ and~$\be_r$ for $r=26$ and projecting onto $X=L^2(\Omega)$ and $X=H_0^1(\Omega)$.}
\end{table}

Comparing \eqref{proyec_L2_er} with \eqref{eq:pointL2}, and taking into account the size of the eigenvalues, we expect the error in the projection, ${\boldsymbol \varepsilon_r}(t)$, in $L^2$ to be smaller if we take $X=L^2(\Omega)$. On the other hand,  comparing now \eqref{eq:ergradH1} with \eqref{eq:ergradL2} (and depending on the size of $\|S_r\|_2$) we could expect the error in the $H^1$ norm to be smaller if we choose $X=H^1_0(\Omega)$. This expected behavior can indeed be observed numerically by checking the darker gray cells in Tab.~\ref{tab:ercont}. We can also observe that the maximum error in the approximation, $\be_r(t)$, in both cases performs similarly. For this reason, in many papers in the literature, there is no preference choice on $X=H^1(\Omega)$. 

In Fig.~\ref{fig:rel-error-L2-H1}, we have represented the errors in the approximation between the POD and the FEM approximation (discontinuous magenta line) and in the projection (continuous blue line) measured in $T/2048$. We observe that both errors are close in most part of the period. 
Since, by definition, the projection is the best approximation in the reduced space to the finite element solution, we can confirm the good performance of the POD method for both choices of $X$ along the period.

\begin{figure}[htbp]
	\begin{center}
		\includegraphics[width=0.49\textwidth]{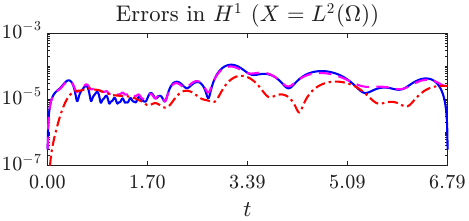} \hfill 
		\includegraphics[width=0.49\textwidth]{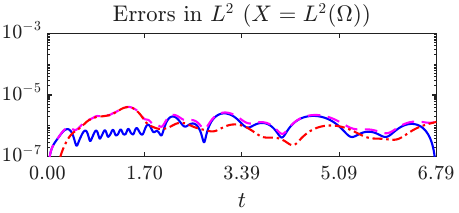} \\[1ex]
		\includegraphics[width=0.49\textwidth]{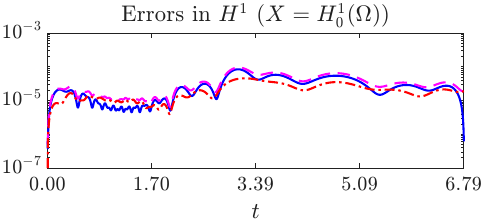} \hfill
		\includegraphics[width=0.49\textwidth]{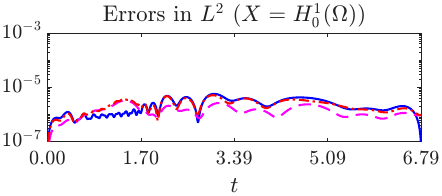}
		\caption{\label{fig:rel-error-L2-H1} Top: relative errors for $X=L^2(\Omega)$. Bottom: relative errors for $X=H^1_0(\Omega)$. Errors $	{\boldsymbol \varepsilon}_r$ (continuous blue line), $\be_r$ (discontinuous magenta line) and $\bz_r - P^r_X \left(\bu_h -p\right)$ (red dotted line), for $M = 128$ and $r = 26$, in $H^1$ (left) and $L^2$ (right) norm.}
	\end{center}
\end{figure}

After checking the similar behavior of both choices of the space $X$, in the following experiments, we fix $X=H_0^1(\Omega)$. Next, we want to study the efficiency of the proposed methods for computing the nonlinear term in the reduced equations described in Sect.~\ref{sec:nonlin}. We compute the POD approximation taking $r=26$ with the FEM structure, the tensor structure and the Discrete Empirical Interpolation Method. 

In Tab.~\ref{tab:compt-time}, we show  the computational cost to get the approximation by integrating in time with the {\sc MATLAB}'s solver \texttt{ode15s}, as before. It is clear that for the two-dimensional model computing the nonlinear term with the FEM structure is inefficient. With DEIM algorithm we get the same errors (as we will show now) but with less computational cost. On the other hand, the computational time needed for the tensor procedure is approximately double than the corresponding to DEIM. Furthermore, the offline computational time needed in the tensor case is huge.

\begin{table}[!t]
	\centering
	\begin{tabular}{lc}
		\hline
		\textbf{Method} & \textbf{Online computational time (s)} \\
		Complete FEM formulation   & 34.1169 \\
		Tensor structure & 0.8388  \\
		DEIM   & 0.4643  \\
		\hline
	\end{tabular}
	\caption{Online computational time to obtain the POD solution for $\beta=2.75$, using $M=128$ and $r=26$. For DEIM, $\tr=m=30$, see Algorithm \ref{alg:deim}.}
	\label{tab:compt-time}
\end{table}

As a conclusion, if the model problem has a nonlinear term with a simple polynomial form, the tensor structure has an exact form and we can reduce the computational cost (of the online phase). However, if the nonlinear term has a more complex structure, we cannot apply the tensor construction and POD-DEIM approximations are sufficiently good and are attained faster than the other alternatives.

For the DEIM implementation, the boundary condition of the nonlinear term in the model $\alpha \cdot \beta$ is subtracted before computing the nonlinear POD basis. It is also necessary to choose the number of modes in that basis ($\tr$) and the number of nodes in the empirical interpolant ($m$). For simplicity, we decide to take $m = \tr$. Taking into account the definition in \eqref{eq:p}, the POD-DEIM approximation is described as $\bu_r^{\text{DEIM}} = \bz_r^{\text{DEIM}} + p$ such that $\bz_r^{\text{DEIM}}:(0,T]\rightarrow \bU^r$ and it solves:
\begin{eqnarray}\label{eq:pod_exdeim}
	(\bz_{r,t}^{\text{DEIM}},v_r)+\nu(\nabla \bz_r^{\text{DEIM}},\nabla v_r)+(\hat{g}(\bz_r^{\text{DEIM}}+p),v_r)=(f,v_r),\quad \forall\ v_r\in \bU^r,
\end{eqnarray}
with $\bz_r(0)=P^r_{H_0^1} \left(\bu^0-p\right)$.

In Tab.~\ref{tab:deim}, we present how the relative error changes for different values of $\tilde{r}$. The projection error in $H^1$ is $8.4756 \times 10^{-5}$, see Tab.~\ref{tab:ercont}, and we want the POD-DEIM approximation to get closer to the projection. So, we set $\tilde{r}=30$ as the relative error in $H^1$ is $9.6525 \times 10^{-5}$.

\begin{table}[!t]
	$$
	\begin{array}{|c|c|c|}
		\hline 
		\tilde{r} &  {\displaystyle \max_{t}}\|\nabla \be_r^{\text{DEIM}}(t) \|_0 & 
		{\displaystyle \max_{t}}\| \be_r^{\text{DEIM}}(t) \|_0
		\\ \hline
		14 & 2.0330 \times 10^{-2} & 1.1267 \times 10^{-3} \\ \hline
		18 & 4.5541 \times 10^{-3} & 4.4254 \times 10^{-4} \\ \hline
		22 & 3.0994 \times 10^{-4} & 1.8701 \times 10^{-5} \\ \hline
		26 & 1.7611 \times 10^{-4} & 9.5249 \times 10^{-6} \\ \hline
		30 & 9.6525 \times 10^{-5} & 4.1698 \times 10^{-6} \\ \hline
	\end{array}
	$$
	\caption{Maximum relative errors for ~$\be_r^{\text{DEIM}}$ in $H^1$ and $L^2$ norm computing POD-DEIM approximation on 2049 equally-distributed points. Taking $r=26$ and $\tr=m=30$.}
	\label{tab:deim}
\end{table}
In Fig.~\ref{fig:femsolution}, the red points represent the values of the $m=30$ nodes in the empirical interpolant. These red points mark the most representative values of the empirical interpolation selected by Algorithm \ref{alg:deim} at different time instances.

In order to see that the relative errors over the period performs as well as the maximum errors of Tab.~\ref{tab:deim}, we can observe Fig.~\ref{fig:erdeim}. In this figure, the $H^1$ norm of the error $\be_r^{\text{DEIM}}$ (discontinuous magenta line), is close to the projection error (blue continuous line). If we compare this figure with Fig.~\ref{fig:rel-error-L2-H1} for $X=H_0^1(\Omega)$, we see that with the DEIM algorithm we achieve the same errors in $H^1$ and $L^2$ norm but with a significant decrease in the computational cost.

\begin{figure}[htbp]
	\begin{center}
		\includegraphics[width=0.49\textwidth]{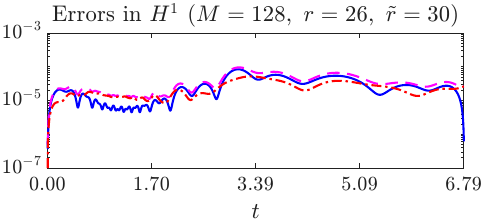} 
		\hfill
		\includegraphics[width=0.49\textwidth]{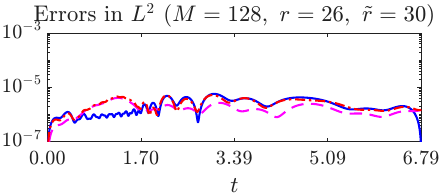}
		\caption{\label{fig:erdeim} $L^2$ (right) and $H^1$ (left) errors ${\boldsymbol \varepsilon_r}$ (continuous blue line), $\be_r^{\text{DEIM}}$ (discontinuous magenta line) and $\bz_r^{\text{DEIM}} - P^r_X \left(\bu_h-p\right)$ (red dotted line), for $M = 128\, , r = 26$. Relative errors measured every $T /2048$ units of time over one period implementing POD-DEIM method using $\tr=30$ and $m=30$.}
	\end{center}
\end{figure}

\subsection{Multi-value parameter model}\label{sec:numparam}
In this section we want to show the performance of the method introduced in \cite{newmethod} for model problems in two spatial dimensions. In \cite{newmethod} numerical experiments in a problem in one spatial dimension can be found. Next, we will show some simulations in the two-dimensional case for understanding the efficiency of this method compared with the standard one when the complexity of the problem increases. 

First, we comment on the choice of the parameters to see the difficulty of the problem. As in Sect.~\ref{sec:numone}, let us fix $\alpha = 1$ and $\nu=0.002$. We want to study the Brusselator model in \eqref{bruss} depending on the parameter $\beta \in [2.6, 2.9]$. We take $\beta$ in that interval as the solution changes significantly when $\beta$ gets larger. The period for $\beta = 2.6$ is $T=6.6352720 s$ and for $\beta=2.9$ is $T=6.9606722 s$ (up to 7 significant digits). We get the finite element approximation of the solution $\bu^\beta_h$ at $\beta = 2.6, \, 2.7, \, 2.8$ and $2.9$. In Fig.~\ref{fig:femsolution-param}, we present the FEM solution at the time instant where the maximum of $ \|\nabla u \|_0^2 + \|\nabla v\|_0^2 $  is attained for each $\beta$. We observe that large variations occur close to the beginning of the period for each parameter (the exact times appearing in the figure). It is also clear that the complexity of the problem (larger gradients) increases with the values of $\beta$. 

\begin{figure}[htbp]
	\centering
	\begin{minipage}{\textwidth}
		\centering
		\includegraphics[width=0.24\linewidth]{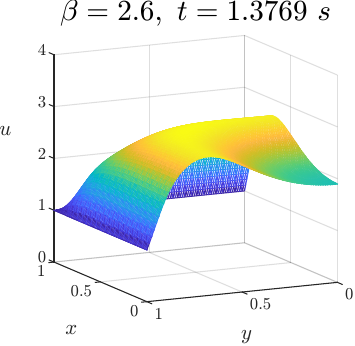}
		\includegraphics[width=0.24\linewidth]{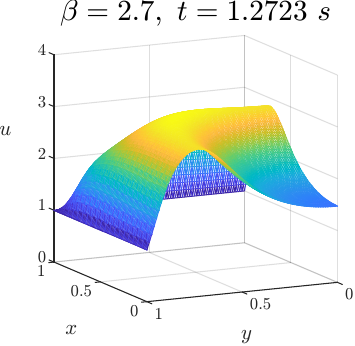} 
		\includegraphics[width=0.24\linewidth]{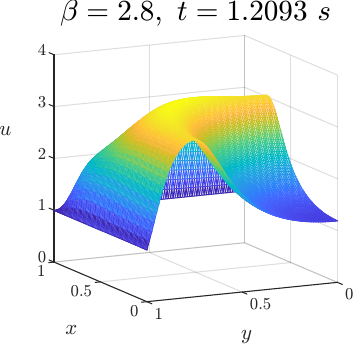}
		\includegraphics[width=0.24\linewidth]{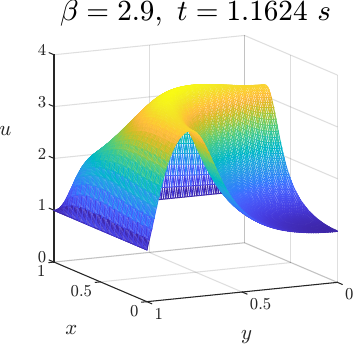}
	\end{minipage}
	
	\vspace{1cm} 
	
	\begin{minipage}{\textwidth}
		\centering
		\includegraphics[width=0.24\linewidth]{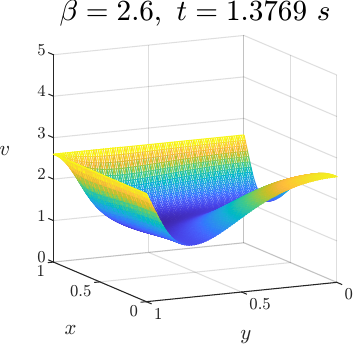}
		\includegraphics[width=0.24\linewidth]{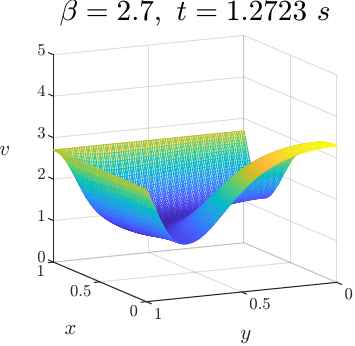} 
		\includegraphics[width=0.24\linewidth]{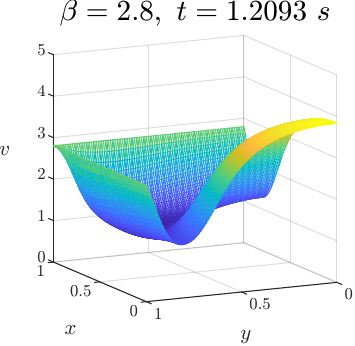}
		\includegraphics[width=0.24\linewidth]{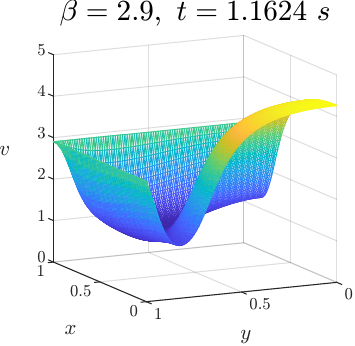}
		
	\end{minipage}
	\caption{Components $u$ (up) and $v$ (down) of the periodic solution at the time where the maximum of $ \|\nabla u \|_0^2 + \|\nabla v\|_0^2 $  is attained for $\beta = 2.6, \, 2.7, \, 2.8$ and $2.9$.}
	\label{fig:femsolution-param}
\end{figure}

For $\beta \in [2.6, 2.9]$, we want to find the POD approximation $\bu_r^{\beta} = \bz_r^{\beta} + p^{\beta}$ with $p^{\beta}$ as in \eqref{eq:p} with the corresponding parameter value $\beta$. The POD basis $\bU^r$ is computed by the new and standard methods introduced in Sect.~\ref{sec:podparam} with the four parameters: $\beta =2.6, 2.7, 2.8$ and $2.9$. The boundary-condition vector $p^{\beta}$ is subtracted to the corresponding snapshots for both methods. The projection space is $X=H_0^1(\Omega)$. The reduced system to solve for $\bz_r^{\beta} :(0,T]\rightarrow \bU^r$ is 
\begin{eqnarray}\label{eq:pod_ex_param}
	(\bz_{r,t}^{\beta},v_r)+\nu(\nabla \bz_r^{\beta},\nabla v_r)+(g(\bz_r^{\beta}+p^{\beta}),v_r)=(f,v_r),\quad \forall\ v_r\in \bU^r,
\end{eqnarray}
with $\bz_r^{\beta}(0)=P^r_{H_0^1} \left(\bu^{\beta}(t_0)-p^{\beta}\right)$. 

As in the previous simulations, we integrate in time \eqref{eq:pod_ex_param} with {\sc MATLAB}'s solver \texttt{ode15s} with a relative tolerance of $10^{-12}$.

In Fig.~\ref{fig:std_nm}, we present singular values, $\sigma_k = \lambda_k^{1/2}$, for the standard and new method taking  $M=32\, , 64\, , 128\, , 256\, , 512$.  Recall from Theorem  \ref{theo:parampoint} that for the new method the error depends on the tail of the eigenvalues. For the standard method, even though the error bounds are not optimal, the error depends also on the tail of the eigenvalues, although with order less than one, see \cite[Theorem 3]{newmethod}. It can be observed that for $M \ge 128$ the results remain unchanged. This means that $M=128$ is sufficient in both POD methods. Overall, for computing the corresponding POD basis it is necessary to compute $N=(M+1)(L+1)=(M+1)4$ elements. It is also clear that the singular values for the new method are considerable larger and the decay is slower than in the standard method. If we consider the discontinuous gray line corresponding to $10^{-3}$, for the standard method it is sufficient to take $r= 36$ whereas for the new method we need to take $r=74$. 

\begin{figure}[htbp]
	\begin{center}
		\includegraphics[width=0.48\textwidth]{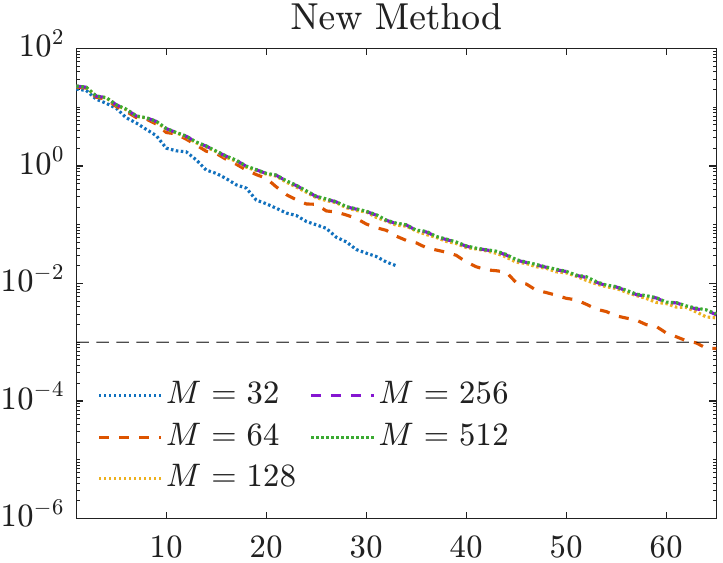}%
		\hfill
		\includegraphics[width=0.48\textwidth]{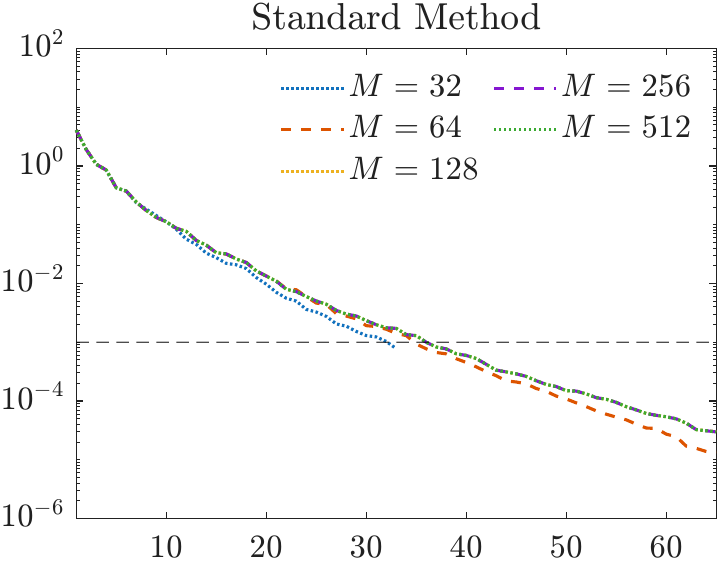}
		\caption{Singular values of the correlation matrix for $M=32, 64, 128, 256, 512$ for the new method (on the left) and the standard method (on the right) taking $\beta = 2.6, 2.7, 2.8$ and $2.9$. The discontinuous black line represents the value $10^{-3}$.}
		\label{fig:std_nm}
	\end{center}
\end{figure}

In Tab.~\ref{tab:er_params}, we want to study the number of necessary modes to get an error in the projection $\approx 10^{-5}$ in the maximum 
$H^1$ norm  for $\beta=2.6$ and $\beta=2.9$. Recall that all errors are relative. We consider two procedures:
\begin{enumerate}
	\item Obtain the POD basis with only one value of parameter $\beta$ (as in Sect.~\ref{sec:numone}) with $X=H_0^1(\Omega)$ and $M=128$.
	\item Obtain the POD basis with four values of parameter $\beta$, as stated before, for the new and standard methods.
\end{enumerate}
We observe that for the first procedure we always need less modes. We can also see that the number of modes increases with the complexity of the problem (we go from $r=21$ for $\beta=2.6$ to $r=34$ for $\beta=2.9$).
On the other hand, we observe that the number of modes in the second case is larger for both methods. Also, the number of modes does not change with the value of the parameter. Finally, concerning the second procedure, we need more modes in the new method than in the standard one. We also consider the projection and approximation errors, defined in~\eqref{error_notation}. We see that in all considered scenarios the approximation error, $\be_r$, is similar to the projection error, ${\boldsymbol \varepsilon}_r$. We can conclude that in both settings the POD methods performs well.

\begin{table}[ht]
	\centering
	\caption{Maximum relative errors for ~$\be_r$ and ${\boldsymbol \varepsilon}_r$ in $H^1$ norm on 2049 equally-distributed points at one period for $\beta=2.6$ and $\beta=2.9$.}
	\label{tab:er_params}
	\renewcommand{\arraystretch}{1.2}
	\begin{tabular}{|l||c|c|c||c|c|c|}
		\hline
		& \multicolumn{3}{|c|}{\textbf{$\beta = 2.6$}} & \multicolumn{3}{|c|}{\textbf{$\beta = 2.9$}}\\ \hline
		&
		$r$ &
		$\max\limits_{t}\|\nabla{\boldsymbol \varepsilon_r}(t)\|_0$ &
		$\max\limits_{t}\|\nabla{\boldsymbol \be_r}(t)\|_0$ &
		$r$ &
		$\max\limits_{t}\|\nabla{\boldsymbol \varepsilon_r}(t)\|_0$ &
		$\max\limits_{t}\|\nabla{\boldsymbol \be_r}(t)\|_0$  \\ \hline
		DQs for one $\beta$ &$21$ & \cellcolor{gray!20}$ 5.6794\times 10^{-5}$& $6.6752 \times 10^{-5}$ & $34$ & \cellcolor{gray!20}$ 4.5470\times 10^{-5}$& $ 5.0174 \times 10^{-5}$ \\ \hline
		Standard method & $54$ & \cellcolor{gray!20}$8.0026 \times 10^{-5}$ & $ 1.5696 \times 10^{-4}$& $54$ & \cellcolor{gray!20}$7.7634\times 10^{-5}$ & $1.4424 \times 10^{-4}$ \\ \hline
		New method & $80$ & \cellcolor{gray!20}$9.1438 \times 10^{-5}$ & $1.3728 \times 10^{-4}$ & $80$ & \cellcolor{gray!20}$7.3756 \times 10^{-5}$& $9.9854 \times 10^{-5}$ \\ \hline
	\end{tabular}
\end{table}

We can compare these results with the numerical simulations for the Brusselator in one spatial dimension in \cite{newmethod}. In that paper, it is considered ~\eqref{bruss} with~ $\Omega = [0,1]$, $\nu=0.01$, $\alpha=1$ and~$\beta\in[2.75,4.25]$, where the periods range from~$T=6.7725$ for $\beta=2.75$ to~$T=9.5949$ for $T=4.25$ (both values rounded to five significant digits). 
For the FEM approximation in~space, quadratic elements were used on a uniform partition of~$[0,1]$ into elements of length~$h=1/50$. For the time integration, {\sc MATLAB}'s command {\tt ode15s} was used with tolerance values $10^{-8}$ and $10^{-11}$ for relative and absolute values, respectively, of the local error.

For the POD method, the value of $M$ is  $M=64$, that is, for every $\beta=2.75, 3.25, 3.75$, $4.25$ there are $65$ snapshots on equally distributed times $t_n^{\beta}$ on each period~$[0,T^\beta]$, and
the corresponding POD basis and POD approximations~$\bu_r$ were computed for the new and standard method for $r=42$.

In this scenario, the conclusion of that article was that the new method performs slightly better than the standard one with the same number of modes. However, in our numerical experiments with higher complexity we cannot conclude the same results. In Fig.~\ref{fig:1d}, we present the singular values of the standard and new approaches for the problem described in \cite{newmethod} for one spatial dimension. We can observe that under that scenario, the singular values have a similar size and the decay is faster in the new method. This explains the better behavior of the new method in the one-dimensional case.

\begin{figure}[ht]
	\begin{center}
		\includegraphics[width=0.48\textwidth]{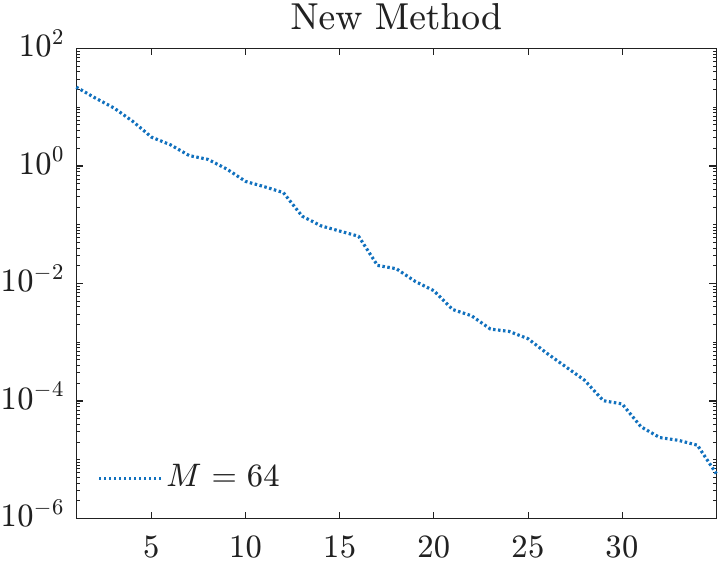}
		\includegraphics[width=0.48\textwidth]{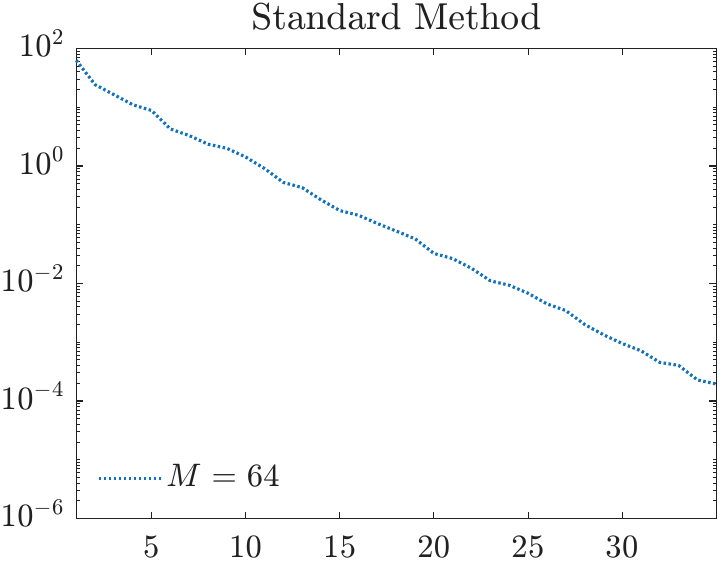}
		\caption{\label{fig:1d}Singular values of the correlation matrix for $M=64$ for the new method (on the left) and the standard method (on the right) taking $\beta = 2.75, 3.25, 3.75$ and $4.25$ in the one-dimensional Brusselator model, see \cite{newmethod}.}
	\end{center}
\end{figure}

To conclude, we want to point out that with both the standard and new method we can get not only approximations to the solutions corresponding to parameters in the data set, but also for those corresponding to any parameter value in the interval, with sufficient accuracy. To check this fact, we take $\beta =2.75$, a parameter out of our sample in the two-dimensional case. We compute the POD approximation in both cases for $r=80$ modes and we use as the initial condition the projection the finite element approximation at the initial time. Recall from Sect.~\ref{sec:numone} that the period in time for this parameter is $T= 6.7857023$s (up to seven significant digits). In Fig.~\ref{fig:out}, we present the relative projection errors in $H^1$ and $L^2$ of the POD approximation with the standard (dark green continuous line) and new method (blue continuous line). We can also see that the approximation errors in $H^1$, the discontinuous orange line for the new method and the discontinuous magenta for the standard, are close to the projection error. That assures the 
good behavior of  both methods for this out of the sample parameter. As for the parameters in the data set in Tab.~\ref{tab:er_params}, the standard method performs better than the new method along the period.

\begin{figure}[ht]
	\begin{center}
		\includegraphics[width=0.49\textwidth]{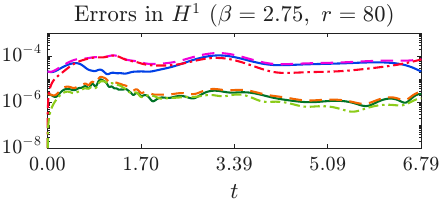}
		\hfill
		\includegraphics[width=0.49\textwidth]{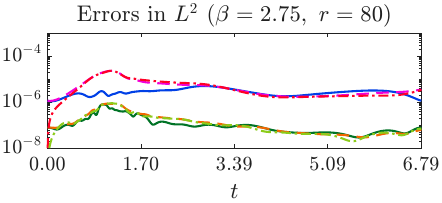}
		\caption{\label{fig:out} $L^2$ (right) and $H^1$ (left) errors ${\boldsymbol \varepsilon_r}$ (continuous dark green line for the standard method, and continuous blue line for the new methd), $\be_r$ (discontinuous orange line for the standard and magenta line for the new method) and $\bz_r^{\beta} - P^r_{H_0^1} \left(\bu_h^{\beta}-p^{\beta}\right)$ (light green dotted line for the standard and red dotted line for the new method), for $M=128,\, r = 80$ in both cases. Relative errors measured every $T /2048$ units of time over one period.}
	\end{center}
\end{figure}

\section{Conclusions}
In this article, we study some computational aspects of reduced-order methods of POD type for evolutionary parametric equations. Firstly, we analyze the two options on the projection space. In the literature, projecting onto $L^2(\Omega)$ is frequent although
it is proved that projecting onto $H^1_0(\Omega)$ the a priori bounds are optimal. We analyze both scenarios and conclude that despite the differences in the error estimates, in practice, both of them perform similarly. 

We also study the computational cost of implementing the nonlinear term in the reduced equations. There are three strategies that we compare to address this issue:
\begin{itemize} 
	\item Computation of the nonlinear term with FEM formulation (in the online part of the code).
	\item Construction of a tensor structure (in the offline part) and the use of $r$ modes in the tensor structure (in the online part of the code).
	\item Use of Discrete Empirical Interpolation Method. It consists on the computation of a POD basis of the nonlinear term and by fixing a number of modes $\tr$, the implementation of DEIM algorithm for selecting significant spatial nodes in the nonlinear part. The nonlinear term is projected onto the original POD basis with a cost based only on $r$ and $\tilde r$.
\end{itemize}
The accuracy of the three procedures is almost the same while the computational cost of the last two methods is significantly smaller. Moreover, the second method only works for polynomial structures of the nonlinear term and the DEIM procedure is faster (around half of the time).

Finally, we add a complementary study in two spatial dimensions to the results presented in \cite{newmethod} (for the one-dimensional case) for the new and standard method, both of them analyzed in Sect.~\ref{sec:podparam}. We conclude that, contrary to the one-dimensional case, the standard method performs better.

\bibliographystyle{abbrv}
\bibliography{ref.bib}
\end{document}